\documentclass[10pt]{amsart}

\usepackage[T1]{fontenc}
\usepackage[utf8]{inputenc}
\usepackage[english]{babel}
\usepackage{csquotes} 
\usepackage[a4paper,margin=3cm]{geometry}
\usepackage{graphicx}
\usepackage[backend=biber,style=numeric,bibencoding=utf8,language=auto,autolang=other,giveninits=true,doi=false,isbn=false,url=false,maxnames=10,sorting=nty]{biblatex}
\usepackage{caption} 
\usepackage[hyperfootnotes=false]{hyperref}
\hypersetup{
	colorlinks = true,
	linkcolor = {blue},
	urlcolor = {red},
	citecolor = {blue}
}

\usepackage[foot]{amsaddr}

\usepackage{footmisc}
\usepackage[nameinlink,capitalise,sort]{cleveref}
\crefname{equation}{}{}
\crefname{enumi}{}{}

\usepackage{amsmath}
\usepackage{amsthm}
\usepackage{amssymb}
\usepackage{esint} 
\usepackage{mathrsfs} 
\usepackage{faktor} 
\usepackage{dsfont} 
\usepackage[dvipsnames]{xcolor}
\usepackage{array}
\usepackage{hhline}
\usepackage{enumitem} 
\usepackage{comment} 
\usepackage[toc,page]{appendix} 
\usepackage{imakeidx} 
\usepackage{xparse} 
\usepackage{mathtools}
\usepackage{fancyhdr} 
\usepackage{ifthen} 
\usepackage{forloop} 
\usepackage{xstring}
\usepackage{tikz}
\usetikzlibrary{babel} 
\usetikzlibrary{cd} 
\usepackage{emptypage} 
\usepackage{listings} 
\usepackage{mathabx} 

\theoremstyle{plain}
\newtheorem{lemma}{Lemma}[section]

\newtheorem{theorem}[lemma]{Theorem}

\theoremstyle{definition}

\theoremstyle{remark}
\newtheorem{remark}[lemma]{Remark}

\numberwithin{equation}{section}

\newcommand{\N}{\mathbb{N}}

\newcommand{\R}{\mathbb{R}}
\newcommand{\T}{\mathbb{T}}
\newcommand{\Z}{\mathbb{Z}}

\newcommand{\pt}{\partial_t}

\newcommand{\norm}[1]{\left\|#1\right\|}

\newcommand{\eps}{\varepsilon}

\newcommand{\grad}{\nabla}

\DeclareMathOperator*{\sign}{sign}
\newcommand{\sgn}[1]{\sign\left(#1\right)}

\DeclareMathOperator*{\supp}{supp}

\newcommand{\dd}{\,\mathrm{d}}
\renewcommand{\d}{\mathrm{d}}
\renewcommand{\div}{\operatorname{div}}
\newcommand{\de}{\partial}

\newcommand{\BMO}{\mathrm{BMO}}

\DeclareMathOperator{\dist}{dist}

\def\intT{\int_{\T^d}}
\def\intR{\int_{\R^d}}

\begin{document}

\title[Optimal regularity for the 2D Euler equations in the Yudovich class]{Optimal regularity for the 2D Euler equations in the Yudovich class}

\author[N.~De Nitti]{Nicola De Nitti}
\address[N.~De Nitti]{EPFL, Institut de Math\'ematiques, Station 8, 1015 Lausanne, Switzerland.}
\email[N.~De Nitti]{nicola.denitti@epfl.ch}

\author[D.~Meyer]{David Meyer}
\address[D.~Meyer]{Universit\"at M\"unster, Institut f\"ur Analysis und Numerik, Orl\'eans-Ring 10, 48149 M\"unster,  Germany.}
\email[D.~Meyer]{dmeyer2@uni-muenster.de}

\author[C.~Seis]{Christian Seis}
\address[C.~Seis]{Universit\"at M\"unster, Institut f\"ur Analysis und Numerik, Orl\'eans-Ring 10, 48149 M\"unster, Germany.}
\email[C.~Seis]{seis@uni-muenster.de}


\begin{abstract}
We analyze the optimal regularity that is exactly propagated by a transport equation driven by a velocity field with a BMO gradient. As an application, we study the 2D Euler equations in case the initial vorticity is bounded. The sharpness of our result for the Euler equations follows from a variation of Bahouri and Chemin's vortex patch example.
\end{abstract}

\maketitle

\section{Introduction and main results}
\label{sec:intro}

Our main motivation in the present work is a regularity question for the two-dimensional Euler equations that describe the motion of a perfect, incompressible, inviscid fluid  (see \cite{MR1867882,MR2338368,MR1245492,MR1340046}). We study their vorticity formulation on the two-dimensional torus\footnote{ We work on the torus $\mathbb T^d$  in order to avoid integrability issues at infinity and obtain global estimates. However, here, as in the previous related works \cite{MR4263701,2203.10860}, the results hold in the space $\mathbb R^d$ as well, at the price of extra work in localizing the estimates appropriately.}, $\T^2 = \R^2 / 2\pi \Z^2$: 
\begin{equation}\label{eq:vorticity}
\begin{cases}
    \partial_t \omega + u\cdot \grad \omega = 0, & t >0, \ x \in \T^2,\\
    \omega(0,x) = \omega_0(x), & x \in \T^2,
\end{cases}
\end{equation}
where $\omega=\omega(t,x)\in\R$ is the scalar vorticity distribution of the ideal fluid, $\omega_0=\omega_0(x) \in \R$ is its initial value, and $u=u(t,x)\in\R^2$ is its (divergence-free) velocity vector field. We recall that the latter can be recovered from the vorticity via the Biot--Savart law, \(u=\nabla^\perp\Delta^{-1} \omega\), which, on the torus, takes the simple form
\begin{equation}\label{eq:bs-torus}
\hat u(t,\eta) = -\frac{i\eta^{\perp}}{|\eta|^2} \hat \omega(t,\eta), \quad t >0, \ \eta \in 2\pi \mathbb{Z}^2,
\end{equation}
where $\hat f = \hat f(\eta)$ denotes the Fourier transform of a function $f=f(x)$, and $\eta^{\perp}$ is the counter-clockwise rotation by ninety degrees of the wave-number $\eta$.

Our goal is to identify the optimal regularity that is  propagated by the nonlinear dynamics \cref{eq:vorticity} under the mere assumption that the vorticity distribution is uniformly bounded, so that
\begin{equation}\label{eq:linfty-vort}
\|\omega(t,\cdot)\|_{L^{\infty}(\T^d)} = \|\omega_0\|_{L^{\infty}(\T^d)}<+\infty, \quad\text{for all}\quad t >0.
\end{equation}
This setting is particularly interesting because $L^\infty$ is a critical space in terms of the vorticity: the velocity field associated with a bounded vorticity distribution via \cref{eq:bs-torus} fails to be Lipschitz continuous in the spatial variable; instead, it is only $\log$-Lipschitz, i.e., there exists $L>0$ (depending on $\omega_0$) such that 
\begin{equation}\label{eq:log-lip}
|u(t,x) - u(t,y)| \le L |x-y|\left(1+\log \frac{4\pi}{|x-y|}\right),\quad \text{for all} \quad t >0, \ x,\,y\in \T^2.
\end{equation}
Yet, it is known since Yudovich's work   \cite{Yudovich} that this $\log$-Lipschitz regularity is sufficient in order to guarantee well-posedness of the Euler equations \cref{eq:vorticity}: there exists one and only one solution belonging to  $L^{\infty}\left([0, +\infty) \times \mathbb{T}^2\right) \cap C\left([0, +\infty) ; L^1\left(\mathbb{T}^2\right)\right)$ and satisfying the weak formulation of \cref{eq:vorticity}, i.e., 
\[
\int_{\mathbb{T}^2} \omega(t,x) \varphi(t,x) \, \mathrm{d} x-\int_{\mathbb{T}^2} \omega_0(x) \varphi(0,x) \, \mathrm{d} x=\int_0^t \int_{\T^2} \omega(s,x) u(s,x) \cdot \nabla \varphi(s,x) \, \mathrm{d} x \, \mathrm{d} s
\]
for all $\varphi \in C^{\infty}\left(\mathbb{T}^2\right)$. 

However, as the advecting field is not Lipschitz continuous, the solution behaves quite poorly with regard to its regularity: In \cite{MR1288809}, Bahouri and Chemin observed that the regularity deteriorates during the evolution. In fact, any solution $\theta=\theta(t,x)\in\R$ to the transport equation
\begin{equation}\label{eq:transport}
\begin{cases}
\partial_t \theta + u \cdot \nabla \theta  = 0,  & t >0, \ x \in \T^d, \\
\theta(0,x) = \theta_0(x), & x \in \T^d, 
\end{cases}
\end{equation}
(in which, for $d=2$, $\theta=\omega$ is, of course, admissible) may have exponentially-decreasing H\"older regularity, i.e., it may only lie in the H\"older space $C^{e^{-\kappa t}}$ at time $t >0$, for some constant $\kappa>0$ depending on the size of the velocity. Such losing estimates that degenerate over time can also be stated in Sobolev spaces: if  $\theta_0 \in W^{s+,\,p}$, with $0<s \le 1$ and $p \ge 1$, then the solution may not belong to the same space, but 
\[
\theta(t,\cdot)\in W^{s(t),\,p}, \quad\text{with}\quad  
s(t) \coloneqq se^{-\kappa t}, \quad t >0.
\]
More recently, in \cite{MR4263701}, Bru\`e and Nguyen revisited the regularity problem for the linear transport equation \cref{eq:transport} (in arbitrary space-dimensions) in the case where the velocity gradient is exponentially integrable, i.e., there exists $\beta \in\R$ such that
\begin{align}\label{eq:exp-intro}
    \sup_{t>0}\int_{\T^d} \exp\left(\beta |\grad u(t,x)|\right)\dd x <+\infty\,.
\end{align}
We stress that such a $\beta$ exists for velocity fields \cref{eq:bs-torus} if the vorticity is bounded\footnote{ The exponential integrability condition \cref{eq:exp-intro} implies that $u$ satisfies a log-Lipschitz estimate (by \cite[Lemma A.1]{MR4263701}). In particular, it is Osgood continuous, i.e.,  \begin{align}
|u(x)-u(y)|\lesssim \eta(|x-y|), 
\end{align}
where $\eta$ is a function such that \begin{align}
\int_0^1\frac{1}{\eta(s)}\, \d s=+\infty.
\end{align}
This implies the uniqueness of the flow associated with $u$ (see \cite{MR2439520}). We also point to \cite{zbMATH07739119} for some extensions of these results to cases of sub-exponential integrability.}, and thus, Bru\`e and Nguyen's findings do also apply to solutions to the Euler equation \cref{eq:vorticity}. Extending the earlier work \cite[Theorem 4]{MR3912727}, they showed that, if  $\theta_0 \in W^{s,\,p}$, with $0<s \le 1$ and $p \ge 1$, then 
\[
\theta(t,\cdot)\in W^{s(t),\,p}, \quad\text{with}\quad  
s(t) \coloneqq  \frac{s}{1+\kappa t}, \quad t >0.
\]

The loss of regularity can be also put on the integrability. Indeed, in \cite{Jeong}, Jeong  produced solutions in the Yudovich class such that the gradient of the vorticity loses integrability continuously in time: $\omega(t,\cdot) \notin W^{1, p(t)}$, with $t \mapsto p(t)$ decreasing continuously  and $1 \leq p(0)<2$. In \cite{2206.14237}, similar results were proven for general Osgood velocity fields. In \cite{MR2231013}, Danchin showed similar estimates with a loss of regularity for the transport equation in critical Besov spaces.

However, outside of the Yudovich class, an instantaneous loss of fractional regularity for \cref{eq:vorticity} may occur, as shown in \cite{Cordoba} for $d=2$ and initial vorticity in $H^s$, with $0 < s < 1$. 

Analogously, if the velocity field is not log-Lipschitz, but merely $W^{1,p}$, (with $1 \le p < \infty)$, the linear transport equation \cref{eq:transport} is well-posed, e.g., in $L^\infty((0,+\infty);L^\infty(\T^d))$, 
owing to DiPerna--Lions' theory of renormalized solutions (see \cite{DiPernaLions89}), but may instantaneously lose any initial Lipschitz and Sobolev regularity, even of fractional order, during its evolution (see \cite{MR3933614}). Yet, some very weak regularity is propagated. First, in \cite{MR2369485}, Crippa and De Lellis observed that a quantitative Lusin--Lipschitz estimate holds for regular Lagrangian flows (in the sense of Ambrosio, see \cite{Ambrosio04}), provided that $p>1$. At the Eulerian level, this result translated into the propagation of a logarithmic Sobolev derivative, as shown in \cite{MR4377866,2203.10860,MR3995045}; see also \cite{Seis17} for dual results. For $W^{1,1}$ velocity fields, the propagation of regularity is a major open problem related to Bressan's mixing conjecture (see \cite{MR2033002})\footnote{ However, we refer to \cite{Bianchini06} for a related positive result in a one-dimensional setting.}; the only known results show that the regular Lagrangian flow is differentiable in measure\footnote{ Further results are available in the special case of two-dimensional, divergence-free, autonomous vector fields; see \cite{BM21,zbMATH07399532}.} (see \cite{MR2044334, zbMATH05228217} and also \cite{MR4514062} for the case of nearly incompressible BV vector fields). 

Our first main result concerns the optimal regularity that is \emph{exactly propagated} by the transport equation \cref{eq:transport} (in any space-dimension $d\ge 2$) in the setting where $\nabla u(t,\cdot)$ is exponentially-integrable: it is at most a derivative of the order $\exp(s\log^{1/2}|\grad|)$ that can be controlled \emph{without loss} during the evolution, where $s>0$ is arbitrarily fixed. More precisely, given $a,\,s \ge 0$, we measure this type of regularity in terms of the Sobolev semi-norm
\begin{equation}\label{eq:sobolev-norm}
\norm{\theta}_{H^{\exp, a,s}(\T^d)}\coloneqq \left(\sum_{\eta \in 2\pi \Z^d\setminus\{0\}} \exp(2s\log^{a}|\eta|)  |\hat \theta(\eta)|^2\right)^{1/2}.
\end{equation}
The space of integrable fuctions for which this norm is finite will be denoted by $H^{\exp,a,s}(\T^d)$.\footnote{ Notice that for $a=1$, this is just the standard semi-norm on the fractional Sobolev space $H^{s} \equiv W^{s,2}$.}

In the interest of brevity, we state our  theorem assuming $\nabla u(t,\cdot) \in \BMO(\T^d)$. Indeed, it is well known that singular integrals/Fourier multipliers of $L^\infty$-functions belong to $\BMO$ (see, e.g., \cite[Chapter IV, Section 4.1]{MR1232192}), not to $L^\infty$, in general, though; yet, $\BMO$ functions are still locally in every $L^p$ space (for $p<\infty$) and even exponentially integrable\footnote{ Owing to John--Nirenberg's inequality (see \cite[Chapter IV, Section 1.3]{MR1232192}), \cref{eq:exp-intro} follows from assuming $u \in L^\infty((0,+\infty); \BMO(\T^2))$.}. However, this assumption can actually be relaxed to assuming that $\nabla u(t,\cdot)$ and some Littlewood--Paley sums of $\nabla u(t,\cdot)$ are exponentially integrable: we point to  \crefrange{ass:v-1}{ass:v-t} below for the precise conditions.

\begin{theorem}[Propagation of regularity for the transport equation]\label{th:transport} 
Let us assume that the velocity $u$ satisfies 
\begin{align}
 \label{ass:div-free-theo}  & \div u = 0, \\
 \label{ass:bmo-theo}  & \grad u \in L^1((0,+\infty); \BMO(\T^d)). 
\end{align}
Then, the (unique) renormalized solution $\theta$ of the Cauchy problem  \cref{eq:transport}, with initial data $\theta_0$ such that $\theta_0 \in L^\infty(\T^d) \cap H^{\exp, a, s}(\T^d)$ and $\fint_{\T^d} \theta_0 = 0$, satisfies the estimate 
\begin{align}\label{eq:claim-regularity}
\norm{\theta(t,\cdot)}_{H^{\exp, a,s}(\T^d)} \le C\, \exp\left(C\int_0^t \norm{\nabla u(\tau,\cdot)}_{\BMO(\T^d)}   \dd \tau\right)\left(\norm{\theta_0}_{H^{\exp, a,s }(\T^d)}+\norm{\theta_0}_{L^\infty(\T^d)}\right),
\end{align}
for all $a \le 1/2$ and all $s>0$, with a constant $C>0$ depending on $a$, $s$, and $d$. 
\end{theorem}

We prove \cref{th:transport} using the approach of \cite{2203.10860}, which quantifies the commutator lemma by DiPerna and Lions (see \cite{DiPernaLions89}) via Littlewood and Paley's decomposition of functions into frequency blocks. 

\begin{remark}[Advection-diffusion equation]
The approach used to prove  \cref{th:transport} extends easily to deal with advection-diffusion equations,  
\begin{equation}
    \label{302}
    \begin{cases}
    \partial_t \theta + u\cdot \nabla \theta = \kappa \Delta\theta, & t >0, \ x \in \T^d, \\
    \theta(0,x) = \theta_0(x), & x \in \T^d,
    \end{cases}
\end{equation}
where  $\kappa> 0$ is the diffusivity constant. In this case, following \cite{2203.10860}, we gain additionally control over the dissipation term
\[
\left(\kappa \int_0^t \norm{\grad \theta(\tau,\cdot)}_{H^{\exp,a,s}(\T^d)}^2\dd \tau \right)^{\frac12},
\]
which can be added on the left-hand side of the regularity estimate \cref{eq:claim-regularity}. 
\end{remark}

Next, we address the question of the sharpness of \cref{th:transport}, which we will answer in the particular case of the Euler equations. For $a > 1/2$, we show that there exists an initial datum $\omega_0\in L^\infty(\T^d)\cap H^{\exp,a,s}(\T^d)$ such that the Euler vorticity  $\omega(t,\cdot)$ does not belong to  $H^{\exp,a,s}(\T^d)$.
\begin{theorem}[Sharpness of the regularity result]\label{th:sharp}
    The result of \cref{th:transport} does not hold for $a > 1/2$. That is, for any $s >0$ and $a>1/2$, there exists an initial datum $\omega_0\in L^\infty(\T^2)\cap H^{\exp,a,s}(\T^2)$ and time-horizon $T>0$ such that the solution $\omega(t,\cdot)$ of the vorticity equation \cref{eq:vorticity} does not belong to $H^{\exp,a,s}(\T^2)$ for all $0 <t< T$. 
\end{theorem}

The construction that allows us to prove \cref{th:sharp} is based on the vortex-patch example used by  Bahouri and Chemin in \cite{MR1288809}. The result does not contradict the above-mentioned losing estimates in $H^s$-spaces, as functions in $H^{\exp,a,s}$ do not necessarily have any fractional regularity.

\begin{remark}[Loss of fractional Sobolev regularity]\label{rk:Hs-bu}
    Analogously, for $s< \frac{1}{2}$, we can also establish an instantaneous blow-up of any $H^{s}$ norm of the solution of \cref{eq:vorticity} (see \cref{rk:fractional-spaces} for more details). 
\end{remark}

Finally, we stress that vector fields  in $W^{\frac{d}{p}+1,\,p}$ satisfy \cref{eq:exp-intro} (see \cite[Theorem 2]{zuazua2002log}). This is precisely the threshold above which the Sobolev embedding and Calder\'on--Zygmund estimates guarantee that $\grad u(t,\cdot)$ is a Lipschitz continuous function. If the initial velocity belongs to $W^{s,p}$, with $s>\frac{d}{p}+1$, then   $\grad u(t,\cdot)$ is Lipschitz continuous, and much better estimates are available.  Local well-posedness of the Euler system in these Sobolev spaces was established first in \cite{MR271984} (for $p=2$) and in \cite{MR344713} (for $p >1$) in a compact domain, and then in \cite{MR481652, MR904939} in the whole space\footnote{ More refined results are available in Besov-type spaces. In \cite{MR1717576,MR1664597}, Vishik constructed global solutions of the two-dimensional Euler equations in Besov space $B_{p, 1}^{\frac{2}{p}+1}$ with $1<p<\infty$ (with regularity index $s=\frac{2}{p}+1$, which is critical in two space-dimensions). For $d \geq 2$, local well-posedness in the critical Besov space $B_{p, 1}^{\frac{d}{p}+1}$ with $1<p<\infty$, was obtained in \cite{MR2072064}; on the other hand, the unique local solvability for $p=\infty$, i.e., in $B_{\infty, 1}^1$ and for $p=1$, i.e., in $B_{1,1}^{d+1}$ was proved in \cite{MR2097579} and  \cite{MR3071222}, respectively.}. These results actually hold globally in time, owing to Beale--Kato--Majda's criterion (see \cite{MR763762}). This does not happen at critical regularity: in \cite{MR3359050,MR4300224},  Bourgain and Li proved ill-posedness of the Euler system in $W^{\frac{d}{p}+1,p}$ (see also \cite{MR4176667, MR3625192,MR3812093,MR4470735} for further developments\footnote{We also refer to \cite{MR2566571} for an earlier ill-posedness result in Besov spaces $B_{r, \infty}^s$, with $s>0$ if $r>2$ and $s>d\left(\frac{2}{r}-1\right)$ if $1 \leq r \leq 2$.}).  

The end-point cases $p=\infty$ (i.e., spaces such as $C^{m, \alpha}$, with $m \ge 1$ and $0<\alpha<1$) were similarly investigated. The local well-posedness in $C^{m, \alpha}$ (with $m \geq 1$ and $0<\alpha<1$) was proven in \cite{MR1545189,gunther1927} and the global-in-time result, for $d=2$,  in \cite{MR1545430,MR1545431}. Although the solution in this case is globally regular, infinite-time growth of their derivatives and small-scale creation appear as well: the upper-bound on the growth of $\nabla u(t,\cdot)$ is double-exponential in time (due to \cite{zbMATH03226068}) and, in \cite{MR3245016}, Kiselev and \v{S}ver\'{a}k indeed produced an example of such growth. On the other hand, ill-posedness was established in \cite{MR3320889,MR4065655} for initial velocities in $C^m$ or $C^{m-1,1}$, with $m \ge 1$ being an integer.

\subsection{Outline}
\label{ssec:outline}

The paper is organized as follows. In  \cref{sec:prelim}, we introduce some tools and definitions that are needed in the proofs.  \cref{th:transport} and \cref{th:sharp} are proven in \cref{sec:proof-hyp} and \cref{sec:proof-counterexample}, respectively. \medskip

\emph{Notation.} We let $C$ denote an implicit constant which is allowed to depend on $a,s,d$ but not on the other parameters unless explicitly mentioned and which is allowed to change values from line to line. We will often write $g\lesssim h$ for two reals $g$ and $h$ if there exists such a constant $C$ depending possibly on $d$, $s$, $a$ such that $f\le Cg$.

\section{Preliminaries}
\label{sec:prelim}

In this section, we introduce some preliminary notions that are needed to prove \cref{th:transport} and \cref{th:sharp}.

\subsection{Littlewood--Paley decomposition}
\label{ssec:lp}

In this section, we recall the \emph{Littlewood--Paley decompositions of periodic functions} (see also \cite{2203.10860} and \cite{MR3243734}).

The \emph{Fourier transform} of a  function $\theta$ on the torus $\T^d \coloneqq  \R^d/2\pi\Z^d$  and its \emph{Fourier representation} are given by
\[
\widehat \theta(\eta) \coloneqq \fint_{\T^d} e^{-i\eta\cdot x} \theta(x)\, \dd x,\quad \theta(x) \coloneqq  \sum_{\eta\in \Z^d} e^{i\eta\cdot x}\widehat \theta(\eta),
\]
respectively, where  $\eta\in \Z^d$ is the wave number and  $x\in \T^d$. 
For a function $\phi$ on the full space $\R^d$, these quantities are defined by 
\[
\widehat \phi(\xi) = \frac1{(2\pi)^{\frac{d}2}} \int_{\R^d} e^{-i\xi\cdot x}\phi(x)\, \dd x,\quad \phi(x) =  \frac1{(2\pi)^{\frac{d}2}}\int_{\R^d} e^{i\xi\cdot x}\widehat \phi(\xi)\, \mathrm{d}\xi,
\]
where now $\xi \in \R^d$ is the frequency and  $x\in \R^d$.

The convolution of $\theta$ and $\phi$ as above is defined as 
\[
(\phi\ast\theta)(x) \coloneqq \int_{\R^d} \phi(x-y)\theta(y)\, \mathrm{d}y = \intT\Big(\sum_{\eta\in\Z^d} \phi(x-z-2\pi\eta)\Big)\theta(z)\, \mathrm{d}z,
\]
(note that, in the first equality, $\theta$ is extended periodically from $\T^d$ to $\R^d$; on the other hand, $\phi \ast \theta$ is periodic in $x$, and thus can be interpreted as a function on $\T^d$).
The Fourier transform of the convolution $\phi \ast \theta$ obeys the  product rule
\[
\widehat{(\phi\ast \theta)}(\eta) = (2\pi)^{d/2}\widehat \phi(\eta)\widehat \theta(\eta), \quad\text{for all }\eta\in \Z^d.
\]

The \emph{Littlewood--Paley decomposition} of a function $\theta$ that has zero mean, i.e., $\hat \theta(0)=0$, is given by 
\[
\theta = \sum_{k\in\N} \theta\ast \phi_k,
\]
where the $\phi_k$'s are a family of rotationally symmetric Schwartz functions that are generated by a Schwartz function $\phi$ whose Fourier transform $\widehat \phi$ is supported in $B_1(0)$ and satisfies  $\widehat{\phi}(\xi)=1$ for $\xi\in \overline{B_{\frac{1}{2}}(0)}$. More specifically, for any $k\in\N$, the $\phi_k$'s are given by  
\[\phi_k\coloneqq 2^{kd}\phi(2^k\cdot)-2^{(k-1)d}\phi(2^{k-1}\cdot),\] so that 
\[
 \sum_{k\in \N} \widehat \phi_k (\xi)=1,\quad \mbox{for } |\xi|\ge 1,
\] holds. We also note that the Fourier transform of $\phi_k$  is supported on a dyadic annulus, 
\begin{equation}
\label{eq:supp-fourier}
\supp \widehat \phi_k \subset B_{2^{k}}(0)\setminus \overline{B_{2^{k-2}}(0)},
\end{equation}
and thus, each $\phi_k$ has overlapping frequency support only with its direct neighbors, which leads to the useful identity
\begin{equation}
\label{27}
\phi_k = \phi_k\ast \left(\phi_{k-1}+\phi_k+\phi_{k+1}\right).
\end{equation}
The above definition of $\phi_k$ can be simply extended to $k=0$ or any other integer.

For any $k\in\N$, we will use the notation $\theta_k\coloneqq \theta*\phi_k$, indicating the part of $\theta$ whose frequencies are concentrated in the $k$-th dyadic annulus. We furthermore consider the \emph{low pass filters}  $\psi_k\coloneqq \phi+\sum_{j=1}^k\phi_j=2^{kd}\phi(2^k\cdot)$ and  define with them the \emph{high-frequency parts} $\theta_k^\geq\coloneqq \theta-\theta*\psi_{k-1} = \sum_{j=k}^{\infty} \theta_j$ and the \emph{low-frequency parts} $\theta_k^{\leq } \coloneqq \theta\ast \psi_k$, so that $\theta  = \theta_k^\leq + \theta_{k+1}^\geq$.

\subsection{Orlicz and Zygmund spaces and their Luxemburg norms}

The \emph{Orlicz space of exponentially integrable functions} is defined as
\begin{align}\label{def:lexp}
    L^{\exp}(\T^d)\coloneqq \left\{f \in L^1(\T^d): \exists \beta >0 \text{ s.t.\ } \int_{\T^d} \exp\left(\beta |f(x)| \right) \dd x < \infty\right\}.
\end{align}
It is a Banach space endowed with the \emph{Luxemburg norm} 
\begin{align*}
    \| f\|_{L^{\exp}(\T^d) } \coloneqq  \inf\left\{h: \int_{\T^d} \exp\left(\tfrac{|f(x)|}{h} \right)  \dd x \le 1 \right\}.
\end{align*}

We shall  define the \emph{logarithmic Zygmund space} $L^p\log L(\T^d)$, for $p \in [1, \infty)$, as
\begin{align}\label{def:llog}
    L^{p}\log L(\T^d)\coloneqq \left\{f \in L^1(\T^d): \, 
    \int_{\T^d} |f(x)|^p \log(| f(x)| +2) \dd x < \infty\right\}, 
\end{align}
which is a Banach space endowed with the Luxemburg norm 
\begin{align}  
    \| f\|_{L^{p}\log L(\T^d)} \coloneqq  \inf\left\{h: \int_{\T^d} \left(\tfrac{|f(x)|}{h}\right)^p \log\left(\tfrac{|f(x)|}{h} +2\right)  \dd x \le 1\right\}.
\end{align}
The completeness of these spaces is stated in \cite[Corollary 2.6]{MR942266}. We refer to \cite[Definition 3.1.1 and 3.2.1]{MR3931352} for further information on this kind of Orlicz spaces and the corresponding Luxemburg norms. In particular, from \cite[Theorem 6.6]{MR194881}, we can deduce the following H\"older-type inequality: if $f \in L^{\exp}(\T^d)$ and $g,h \in L^2\log L(\T^d)$, we have that $fgh \in L^1(\T^d)$ and 
\begin{align}\label{eq:hi}
\| fgh\|_{L^1(\T^d) } \lesssim  \|f\|_{L^{\exp} (\T^d)}\|g\|_{L^2\log L (\T^d)} \|h\|_{L^2\log L (\T^d)}. 
\end{align}

The Luxemburg norm on $L^p\log L$ can be more conveniently estimated from above as follows.

\begin{lemma}[Estimate of the Luxemburg norm]\label{lm:lux-int} For all $1\leq p<\infty$, we have 
\begin{align}
    \|f\|_{L^p\log L(\T^d)} \lesssim \max\left\{1, \left(\int_{\T^d} |f|^p\log(|f|+2) \dd x\right)^{1/p} \right\} .
\end{align}
\end{lemma}

\begin{proof}
Let $\lambda^p=\int_{\T^d}|f|^p\log(2+|f|)\dd x$. The function $h(x)\coloneqq |x|^p\log(|x|+2)$ is monotone, hence it suffices to show that \begin{align}
\int_{\T^d}h\left(\frac{f}{\max\{1,\lambda\}}\right)\dd x\leq 1.
\end{align}
If $\lambda\leq 1$ this is clearly the case;  otherwise, it holds that \begin{align}
\int_{\T^d}h\left(\frac{f}{\lambda}\right)\dd x\leq \frac{1}{\lambda^p}\int_{\T^d}h(f)\dd x=1.
\end{align}
\end{proof}

\subsection{Besov spaces}

The Sobolev space $H^{\exp,a,s}(\T^d)$, defined by the (semi-)norm in \cref{eq:sobolev-norm}, can be characterized as a Besov space. 

We define the Besov spaces $B^{\exp,a,s}(\T^d)$ and $\mathbf{B}^{\exp,a,s}(\T^d)$ of mean-free functions, endowed with the norms 
\begin{align}
 \norm{\theta}_{{B}^{\exp,a,s}(\T^d)}&\coloneqq \left(\sum_{k \geq 1}  \exp\left(2sk^a\right)\norm{\theta_k}_{L^2(\T^d)}^2\right)^\frac{1}{2},\\
\label{eq:besov-filtered-n}\norm{\theta}_{\mathbf{B}^{\exp,a,s}(\T^d)}&\coloneqq \left(\sum_{k \geq 1} k^{a-1} \exp\left(2sk^a\right)\norm{\theta_k^\ge}_{L^2(\T^d)}\right)^\frac{1}{2},
\end{align}
respectively.

\begin{lemma}[Besov-norm equivalence]\label{lm:besov-equiv}
The following function spaces coincide for $a\in (0,1]$: 
\[
H^{\exp,a,s}(\T^d) \equiv B^{\exp,a,s \log^a 2}(\T^d) \equiv \mathbf{B}^{\exp,a,s \log^a 2}(\T^d).
\]
\end{lemma}

\begin{proof} We set $\tilde s = s\log^a 2$. From the (almost-)orthogonality of the Littlewood--Paley decomposition and the definition of the Sobolev spaces, we deduce that
\[
\|\theta\|_{H^{\exp,a,s}(\T^d)} \lesssim \left(\sum_{k\ge 1} \|\theta_k\|^2_{H^{\exp,a,s}(\T^d)}\right)^{1/2} = \left(\sum_{k\ge 1} \sum_{\eta\not=0} \exp\left(2s\log^a |\eta|\right) |\hat \theta_k(\eta)|^2\right)^{1/2}.
\]
By our choice of the projection $\phi_k$, the Fourier support of $\theta_k$ is restricted to wave numbers $\eta\in2\pi\Z^d$ with $|\eta|\in \left(2^{k-2},2^k\right)$, so that $\exp\left(2s\log^a |\eta|\right) \le \exp\left(2 \tilde sk^a\right)$ in each summand. The embedding $B^{\exp,a,\tilde s}(\T^d) \subset H^{\exp,a,s }(\T^d)$ follows.

In order to prove the reverse inclusion, we similarly observe that
\[
(k\log 2)^a \le \left(\log(4) +\log|\eta|\right)^a \le \log^a(4) + \log^a |\eta|
\]
in the wave number  support of $\theta_k$, and the last inequality holds true because $a\in(0,1]$. It thus follows from the definition of the Besov norms and Plancherel's equality that
\[
\|\theta\|_{B^{\exp,a,\tilde s}(\T^d)}^2  = \sum_{k\ge 1} \exp\left(2\tilde s k^a\right) \sum_{\eta\not=0}|\hat \theta_k(\eta)|^2  \lesssim  \sum_{\eta\not=0} \exp\left(2s\log^a|\eta|\right) \sum_{k\ge 1} |\hat \theta_k(\eta)|^2.
\]
It remains to notice that for any fixed $\eta$, there exists a unique $k\in\N$ for which
\[
\sum_{k\ge 1}  |\hat \theta_k(\eta)|^2 =|\hat \theta(\eta)|^2\sum_{k\ge 1}  |\hat \phi_k(\eta)|^2 = |\hat \theta(\eta)|^2 \left(|\hat \phi_{k}(\eta)|^2 + |\hat \phi_{k+1}(\eta)|^2+|\hat \phi_{k+2}(\eta)|^2\right).
\]
Moreover,  Hausdorff--Young's inequality and the scaling property of the projectors guarantee that
\[
\|\hat \phi_k\|_{L^{\infty}(\R^d)} \le \|\phi_k\|_{L^1(\R^d)} = \|\phi_1\|_{L^1(\R^d)} \lesssim 1.
\]
A combination of the previous estimates gives $H^{\exp,a,s}(\T^d)\subset B^{\exp,a,\tilde s}(\T^d)$.

The equality of the two Besov spaces follows from the facts that \begin{align}
\sum_{k\geq j+1}\norm{\theta_k}_{L^2(\T^d)}^2\lesssim \norm{\theta_k^\ge}_{L^2(\T^d)}^2\lesssim \sum_{k\geq j-1}\norm{\theta_k}_{L^2(\T^d)}^2
\end{align}
by the almost orthogonality and because \begin{align}
    \sum_{j=1}^kj^{a-1}\exp(2\tilde{s} j^a)\approx \exp(2\tilde{s} k^a)
\end{align}
as one sees from $\frac{\mathrm{d}}{\mathrm{d}k} \exp(2\tilde{s} k^a)\approx k^{a-1}\exp(2\tilde{s} k^a)$ and comparing with the integral.
\end{proof}

We will need two interpolation-type estimates, variants of which were derived earlier in \cite{2203.10860}.

\begin{lemma}\label{lm:1}
For any $b\le -a/2$, it holds that
\[
 \Big\|{\sup_{k \ge 1} k^{b}\exp\left(c k^a\right) |\theta_k^\geq|}\Big\|_{L^2\log L(\T^d) }  \lesssim   \|\theta\|_{B^{\exp,a,c}(\T^d)} +\norm{\theta}_{L^\infty(\T^d)}.
 \]
\end{lemma}

\begin{proof}
We may assume \begin{equation}
    \norm{\theta}_{L^\infty}= 1\label{5}
\end{equation} which is not restrictive as both sides are homogenous.
We introduce  $n(k) \coloneqq k^b \exp\left(ck^a\right)$ for notational convenience. 

Applying the estimate of the Luxemburg norm in \cref{lm:lux-int} and using that both occurring suprema are optimized at the same index,  we get 
\begin{align*}
 \Big\|{\sup_{k \ge 1} n(k) |\theta_k^\geq|}\Big\|_{L^2\log L(\T^d) }^2 &\lesssim    \int_{\T^d} \left(\sup_{k\ge 1} n(k)^2  |\theta_k^\geq |^2\right)\log\left(\sup_{k\ge 1} n(k) |\theta_k^\geq |+2\right) \dd x +1\\
    &=    \int_{\T^d} \sup_{k\ge 1}\left(   n(k)^2  |\theta_k^\geq |^2\log\left(n(k)  |\theta_k^\geq |+2\right) \right)\dd x+1.
\end{align*}
Since $n(k)> 1$ and 
\[
\|\theta_k^{\ge}\|_{L^{\infty}(\T^d)} \le \left(1+\|\psi_{k-1}\|_{L^1(\R^d)}\right) \|\theta\|_{L^{\infty}(\T^d)}  =\left(1+\|\psi_{1}\|_{L^1(\R^d)}\right) \|\theta\|_{L^{\infty}(\T^d)} \lesssim 1,
\]
as a consequence of Young's convolution inequality and assumption \cref{5}, we  have 
\[\log\left(n(k)|\theta_k^{\ge}| +2\right)\le \log n(k) + \log (|\theta_k^{\ge}| + 2)\lesssim \log n(k). \] 
Therefore  the right-hand side is controlled by
\begin{equation}
    \label{2}
  \Big\|{\sup_{k \ge 1} n(k) |\theta_k^\geq|}\Big\|_{L^2\log L(\T^d) }^2 
\lesssim   \int_{\T^d} \sup_{k\ge 1} n(k)^2 \log(n(k))  |\theta_k^\geq |^2  \dd x   +1 .
\end{equation}
To further bound the right-hand side,  we compute 
\begin{align*}
    \int_{\T^d} \sup_{k\ge 1}  n(k)^2 \log (n(k)) |\theta_k^\geq  |^2 \dd x  
    &\lesssim \int_{\T^d}\sup_{k\ge 1}  \left(1+  \sum_{l=1}^k\frac{\mathrm{d}}{\mathrm{d}l} \left(n(l)^2 \log n(l)\right)   \right)|\theta_k^\geq |^2 \dd x \\
    & \lesssim  \left(1+ \sum_{l\ge 1} \frac{\mathrm{d}}{\mathrm{d}l} \left(n(l)^2 \log n(l)\right) \right) \int_{\T^d}   \sup_{k\ge l}   |\theta_k^\geq |^2 \dd x .
\end{align*}
As established in the proof of   \cite[Lemma 2]{2203.10860}, the integral term can be bounded by $\|\theta_{l-2}^{\ge}\|_{L^2(\T^d)}^2$, and thus, we find that
\[
    \int_{\T^d} \sup_{k\ge 1}  n(k)^2 \log (n(k))|\theta_k^\geq  |^2 \dd x  
\lesssim \|\theta\|_{\mathbf{B}^{\exp,a,c}(\T^d)}^2,
\]
provided that
\[
 \frac{\d}{\d k} \left(n(k)^2 \log n(k)\right) \lesssim \frac{\mathrm{d}}{\mathrm{d}k}(\exp(c(k-2)^a)).
\]
The latter holds true as long as $b\le -a/2$.
In total, we have derived the desired estimate.
\end{proof}

\begin{lemma}\label{lm:interp1}
Let $a\in(0,1)$,  $b\le -a$ and $c\in\R_+$. Let $\chi_k = 2^{kd}\chi(2^k)$ be a family of mean-free Schwartz functions for which $\widehat \chi_1$ is compactly supported.  Then the following inequality holds for any  $\theta \in L^\infty(\T^d) \cap B^{\exp,a, c}(\T^d)$:
\begin{align}
    \Bigg\| \Bigg(\sum_{k \ge 1} k^b \exp(2ck^a) |\theta \ast \chi_k|^2 \Bigg)^{1/2} \Bigg\|_{L^2\log L(\T^d)} \lesssim  \|\theta\|_{{B}^{\exp,a,c}(\T^d)} +\norm{\theta}_{L^\infty(\T^d)} .
 \end{align}
\end{lemma}

To prove \cref{lm:interp1}, we need the following preliminary result. 

\begin{lemma}\label{lm:j}
Let us define
\[
\alpha(k) \coloneqq k^b\exp\left(2ck^a\right),
\]
 for $k\in\N$, $a\in(0,1]$, $b\in\R$, and $c\in\R_+$. Let  $\rho:\R_{+}\to\R_{+}$ be defined by $\rho(s) \coloneqq  s\log(s+2)$. Then
\begin{align*}\label{interpolation}
    \rho\left(\sum_{k\ge 1} \alpha(k)|x_k|^2 \right) \lesssim \sum_{k\ge 1} \rho\left(\alpha(k) \right) |x_k|^2,
\end{align*}
for any sequence $\{x_k\}_{k\in\N}$ such that $\sup_k |x_k|\le 1$.
\end{lemma}

\begin{proof}It is enough to show the statement for partial sums over the range $k\in\{0,\dots,K\}$. We do that via induction over $K$. 

The base case, $K=1$, amounts to 
\begin{align*}
    \rho(\alpha(1)|x_1|^2) \lesssim \rho(\alpha(1))|x_1|^2,
\end{align*}
which follows from the bound on $x_1$.

For the inductive step, we assume  that
\begin{align*}
    \rho\left(\sum_{k= 1}^{K-1} \alpha(k)|x_k|^2 \right) \lesssim \sum_{k= 1}^{K-1} \rho\left(\alpha(k) \right) |x_k|^2,
\end{align*}
holds, uniformly in $K$, and we want to show that 
\begin{align*}
  \rho\left(\sum_{k=1}^{K-1} \alpha(k)|x_k|^2 + \alpha(K)|x_K|^2 \right) - \rho\left(\sum_{k=1}^{K-1} \alpha(k)|x_k|^2\right)    \lesssim |x_K|^{2} \rho(\alpha(K)),
\end{align*}
holds, again, uniformly in $K$.
Using the convexity of $\rho$, we notice that the latter would follow from the bound
\begin{align*}
    \rho'\left(\sum_{k=1}^{K} \alpha(k)  |x_k|^2\right) \alpha(K)   \lesssim \rho(\alpha(K)).
\end{align*}
By the definition of $\rho$ and the boundedness of $x_k$, it is enough to show that
\begin{equation}
\label{4}
\log\left(\sum_{k=1}^{K} \alpha(k)  +2 \right)\lesssim \log(\alpha(K)+2).
\end{equation}

To prove \cref{4}, we observe that $\alpha(k) \lesssim \frac{\dd}{\dd k} \left(k^{1-a} \alpha(k)\right)$, and thus
\[
\sum_{k=1}^K \alpha(k)\lesssim \int_1^K \alpha(k)\dd k \lesssim K^{1-a} \alpha(K).
\]
Taking the logarithm then yields
\[
\log \left(\sum_{k=1}^K \alpha(k)\right) \lesssim \log \left( K^{1-a} \alpha(K)\right) \lesssim \log(K) + \log(\alpha(K)) \lesssim \log(\alpha(K)),
\]
which is \cref{4} to leading order.
\end{proof}

\begin{proof}[Proof of \cref{lm:interp1}]
We may again assume that $\norm{\theta}_{L^\infty}\ll 1$ by homogenity so that $\norm{\theta*\chi_k}_{L^\infty}<1$ for all $k$. 

As a first step, we use \cref{lm:lux-int} and the definition of $\rho$, followed by \cref{lm:j}, to deduce 
\begin{align*}
    \Bigg\| \Bigg(\sum_{k \ge 1} \alpha(k) |\theta \ast \chi_k|^2 \Bigg)^{1/2} \Bigg\|_{L^2\log L(\T^d)}^2 &\lesssim \int_{\T^d}  \rho\Bigg(\sum_{k \ge 1} \alpha(k) |\theta \ast \chi_k|^2 \Bigg) \dd x+1\\
    &\lesssim \sum_{k \ge 1}     \rho(\alpha(k)) \|\theta \ast \chi_k\|_{L^2(\T^d)}^2  +1.
\end{align*}
We estimate the right-hand side above as follows: we split it into phase blocks, using the almost  orthogonality of the Littlewood--Paley decomposition
\[
 \sum_{k \ge 1}     \rho(\alpha(k)) \|\theta \ast \chi_k\|_{L^2(\T^d)}^2  \lesssim \sum_{j\ge 1}   \sum_{k \ge 1}     \rho(\alpha(k)) \|\theta_j \ast \chi_k\|_{L^2(\T^d)}^2   ,
\]
and use Plancherel's theorem to compute 
\begin{align*}
\|\theta_j\ast \chi_k\|_{L^2(\T^d)}^2 = \sum_{\eta\in\Z^d} |\hat \theta_j(\eta)|^2 |\hat \chi_k(\eta)|^2 \le \sum_{\eta\in\Z^d} |\hat \theta_j(\eta)|^2 \sup_{\eta\in B_{2^j}(0)\setminus B_{2^{j-2}}(0)}|\hat \chi_k(\eta)|^2   .
\end{align*}
Since $\widehat \chi_1$ is compactly supported, we can find $R>0$ such that $\supp \widehat{\chi_1} \subset B_R(0)$. The supremum above is thus  non-zero only if $2^kR\geq 2^{j-2}$; then, recalling the fact that $\chi_k$ is mean-free, we deduce  $\widehat{\chi_k}(0)=0$ and the bound 
\begin{align}
|\widehat{\chi_k}(\eta)| \le 2^j |\grad \widehat{\chi_k}(\eta)| = 2^{j-k} |\grad \widehat{\chi_1}(2^{-k}\eta)| \le 2^{j-k}\|\grad\widehat{\chi_1}\|_{L^{\infty}(\T^d)} \lesssim 2^{j-k},
\end{align}
for any wave number  $\eta \in B_{2^j}(0)\setminus B_{2^{j-2}}(0)$.  Using Plancherel's theorem once more, it thus follows with $k=\ell+j$ that 
\[
 \sum_{k \ge 1}     \rho(\alpha(k)) \|\theta \ast \chi_k\|_{L^2(\T^d)}^2  \lesssim \sum_{j\ge 1}   \sum_{\ell\ge \lfloor -2- \log_2 R\rfloor} \rho(\alpha(\ell+j)) 2^{-\ell} \|\theta_j\|_{L^2(\T^d)}^2  .
\]
We evaluate the sum over $\ell$. By the definitions of $\rho$ and $\alpha(k)$ and by concavity ($a<1$),  we have that
 \begin{align*}
&\sum_{\ell\ge \lfloor -2- \log_2 R\rfloor} \rho(\alpha(\ell+j)) 2^{-\ell}  \\ 
& \le \sum_{\ell\ge \lfloor -2- \log_2 R\rfloor} \exp\left(2c(\ell+j)^a\right) (\ell+j)^{a+b} 2^{-\ell} \\
 &\lesssim \sum_{\ell\ge \lfloor -2- \log_2 R\rfloor}\exp\left(2cj^a+ 2caj^{a-1}\ell \right)2^{-\ell}j^{a+b}\left(1+\frac{\ell}{j}\right)^{a+b}\\
&\lesssim \exp(2cj^a)j^{a+b}\sum_{\ell\ge \lfloor -2- \log_2 R\rfloor}\exp(2caj^{a-1}\ell)2^{-\ell}\left(1+\frac{\ell}{j}\right)^{a+b}.
\end{align*}
Now, if $j$ is large enough, say $2 ca j^{a-1} \le \log 2$, we find a converging geometric series,
\begin{align*}
&\sum_{\ell\ge \lfloor -2- \log_2 R\rfloor}\exp(2caj^{a-1}\ell)2^{-\ell}\left(1+\frac{\ell}{j}\right)^a  = \sum_{\ell\ge \lfloor -2- \log_2 R\rfloor}\exp((2caj^{a-1}-\log2)\ell) \left(1+\frac{\ell}{j}\right)^a ,
\end{align*} 
so that
\[
\sum_{\ell\ge \lfloor -2- \log_2 R\rfloor} \rho(\alpha(\ell+j)) 2^{-\ell} \lesssim \exp(2cj^a)j^{a+b}.
\]
For smaller $j$ the same inequality holds, because an estimate of the form
\begin{align*}
\sum_{\ell\ge \lfloor -2- \log_2 R\rfloor} \exp(2c(\ell+j)^a) (\ell+j)^{a+b} 2^{-\ell}&= \sum_{\ell\ge \lfloor -2- \log_2 R\rfloor} \exp(2c(\ell+j)^a-(\log 2) \ell) (\ell+j)^{a+b}   \lesssim 1,
\end{align*}
with an implicit constant  dependent on $j$ is always true if $a<1$. The statement now follows from a combination of the previous estimates.
\end{proof}

\subsection{Calderon spaces}

We will make use of suitable \emph{Calderon norms} that give a more ``pointwise'' description of our regularity classes. Let $a\in [0,1]$, $b\in \R$ and $s>0$ be given. We say that a function  $\theta\in L^2(\T^d)$ belongs to  $ X_{a,s, b}(\T^d)$ if there is a function $g$ such that \begin{align}
\label{def g} &|\theta(x)-\theta(y)|\leq  \log_2^{-b} \left(1+\frac{1}{|x-y|} \right) \exp\left(-s \log_2^a\frac{1}{|x-y|}\right)(g(x)+g(y)),\\ 
&\text{for a.e.} \ x,\,y\in\T^d. \notag
\end{align}
 In this case, we set \begin{align}
\norm{\theta}_{X_{a,s,b}(\T^d)}\coloneqq \inf\{\norm{g}_{L^2(\T^d)}:\: \text{\cref{def g} holds}\}
\end{align}

Let us establish embedding estimates between the spaces  $X_{a,s,b}$ and the Besov spaces $B^{\exp,a,s}(\T^d)$.

\begin{lemma}[Embedding] \label{pointwise desc}
Let $b>1$, $a \in (0,1]$, and $s >0$. Then for all $\theta$ we have 
\begin{align}
\norm{\theta}_{X_{a,s,-b}(\T^d)}\lesssim \norm{\theta}_{B^{\exp,a,s}(\T^d)}\lesssim \norm{\theta}_{X_{a,s,b/2}(\T^d)}.
\end{align}
\end{lemma}

\begin{proof}
We start with the first inequality and let $\theta\in B^{\exp,a,s}(\T^d)$. Given two different points $x,\,y\in \T^d$, we choose $l\in\N$ such that $2^{-l} \le |x-y| < 2^{-l+1}$. 
We split $\theta$ into low and high frequency parts, $\theta=\theta_{l}^\leq+\theta_{l+1}^{\ge}$, as defined above. Then, by the almost-orthogonality of the Littlewood--Paley decomposition (see \cite[Section 6.3.3]{MR3243734}) and the definition of the Besov norm, it holds that
\begin{align}
\norm{\theta_{l+1}^\ge}_{L^2(\T^d)}\lesssim \left( \sum_{k\ge l+1} \|\theta\|_{L^2(\T^d)}^2\right)^{1/2}  \lesssim \exp(-sl^a)   \norm{\theta}_{B^{\exp,a,s}(\T^d)}.\label{bd high fre}
\end{align}
In view of our choice for $l$, we thus have for the high-frequency part the estimate
\[
|\theta_{l+1}^\geq(x) - \theta_{l+1}^\geq(y)| \le |\theta_{l+1}^\geq(x)|+|  \theta_{l+1}^\geq(y)| \le \exp\left(-s\log_2^a\frac1{|x-y|}\right)\left(  g_{\ge}^l(x) +  g_{\ge}^l(y)\right),
\]
with $g_{\ge}^l \coloneqq \exp(sl^a) |\theta_{l+1}^{\ge}|$ being controlled in $L^2$ by $\|\theta\|_{B^{\exp,a,s}(\T^d)}$. 

For the low-frequency part, we use the Morrey-type inequality (cf. \cite[Appendix A]{MR2369485}): 
\begin{equation}\label{20}
|\theta_l^{\le}(x)-\theta_l^{\le}(y)| \lesssim |x-y|\left(M\grad\theta_l^{\le}(x) + M\grad\theta_l^{\le}(y)\right),
\end{equation}
where $M$ is the Hardy--Littlewood maximal function on the torus. We now invoke the Hardy--Littlewood maximal inequality (see \cite[Theorem 6.3.2]{MR3243734}) and the almost-orthogonality of the Littlewood--Paley decomposition to bound
\[
\| M\grad\theta_l^{\le}\|_{L^2(\T^d)} \lesssim \|\grad \theta_l^{\le}\|_{L^2(\T^d)} \lesssim \left(\sum_{k\le l} \|\grad\theta_k\|_{L^2(\T^d)}^2\right)^{1/2}.
\] 
Due to the estimate $\norm{\grad\theta_k}_{L^2(\T^d)}\lesssim 2^k \norm{\theta_k}_{L^2(\T^d)}$, and the monotonicity of $k\mapsto 2^k \exp(-sk^a)$ for sufficiently large $k$, the latter implies
\[
\| M\grad\theta_l^{\le}\|_{L^2(\T^d)}  \lesssim 2^l \exp(-sl^a)   \|\theta\|_{B^{\exp,a,s}(\T^d)}.
\]
Combining this and our choice of $l$, from \cref{20}, we deduce 
\[
|\theta_l^{\le}(x) - \theta_l^{\le}(y)|\lesssim \exp\left(-s \log_2^a \frac1{|x-y|}\right) \left(g_{\le}^l(x) + g_{\le }^l(y)\right),
\]
with $g_{\le}^l : = 2^{-l} \exp(sl^a) M\grad\theta_l^{\le}$ being controlled in $L^2$ by $\|\theta\|_{B^{\exp,a,s}(\T^d)}$.

Combining the estimates for the low and high-frequency parts via the triangle inequality, we arrive at
\[
|\theta(x) - \theta(y)|\lesssim \exp\left(-s \log_2^a \frac1{|x-y|}\right) \left(g^l(x) + g^l(y)\right),
\]
where $g^l = g_{\ge}^l+ g_{\le}^l$ is a function whose $L^2$ norm is controlled by $\|\theta\|_{B^{\exp,a,s}(\T^d)}$ uniformly in $l$. In order to construct a function $g$ on the right-hand side that is independent of $x$ and $y$, we set
\[
g = \sum_{l\ge1} l^{-b} g_l
\]
with $b>1$ in order to ensure the summability: \[
\|g\|_{L^2(\T^d)} \le \sum_{l\ge 1} l^{-b} \|g_l\|_{L^2(\T^d)} \lesssim \left(\sum_{l\ge 1} l^{-b}\right) \|\theta\|_{B^{\exp,a,s}(\T^d)}\lesssim \|\theta\|_{B^{\exp,a,s}(\T^d)}.
\]
With this, $\theta$ satisfies the bound \cref{def g} as $|\log|x-y||\sim l$ and thus, we obtain
\[
\|\theta\|_{X_{a,s,-b}(\T^d)} \lesssim \|\theta\|_{B^{\exp,a,s}(\T^d)},
\]
for any $b>1$.

For the other inclusion, we pick $\theta\in X_{a,s,b}$ and  first note that \begin{align}
\norm{\theta(\cdot+h)-\theta(\cdot)}_{L^2(\T^d)}\le2 \log_2^{-b}\left(1+ \frac1{|h|} \right)\exp\left(-s\log_2^a\frac1{|h|}\right)\norm{\theta}_{X_{a,s,b}(\T^d)}
\end{align}
for any $h\in \R^d$. We observe that the left-hand side can be rewritten via Plancherel's theorem,
\[
\|\theta(\cdot +h) - \theta(\cdot)\|_{L^2(\T^d)}^2 = \sum_{\eta\in\Z^d} |1-e^{i\eta\cdot h}|^2|\hat \theta(\eta)|^2 ,
\]
and we notice that $|1-e^{i\eta\cdot h}| \gtrsim 1$ for $|\eta\cdot h|\in (1/2,1]$. In particular, choosing $h=2^{-l}e_j$ for each standard basis vector $e_j$, and summing over $j$, we get that
\[
\left(\sum_{|\eta|\in(2^{l-1},2^{l}]} |\hat \theta(\eta)|^2\right)^{1/2} \lesssim \log_2^{-b}\left(1+2^{l}\right) \exp\left(-sl^a\right) \|\theta\|_{X_{a,s,b}(\T^d)},
\]
which implies
\[
 \sum_{|\eta|\in(2^{l-1},2^{l}]} \exp\left(2s (\log^{-a} 2)\log^a|\eta|\right) |\hat \theta(\eta)|^2  \lesssim \log_2^{-2b}\left(1+2^{l}\right) \|\theta\|_{X_{a,s,b}(\T^d)}^2.
\]
It remains to sum over $l\in\N$ and observe that the series over $l^{-2b}$ is convergent for $b>1/2$ to  conclude that
\[
\|\theta\|_{H^{\exp,a,s\log^{-a}2}(\T^d)} \lesssim \|\theta\|_{X_{a,s,b}(\T^d)}.
\]
The desired statement then follows via the equivalence in \cref{lm:besov-equiv}.
\end{proof}

\subsection{Evolution of \texorpdfstring{log\textsuperscript{2}-H\"older}{log2-H\"older} norms by log-Lipschitz vector fields}

We say that a function $\theta:\T^d \to \R$ is \emph{log\textsuperscript{2}-H\"older continuous}, if
\[
[\theta]_{C^{\log^2}(\T^d)} \coloneqq \sup_{|h|\le 1/2} |\log|h||^2 \|\theta(\cdot +h)-\theta\|_{L^{\infty}(\T^d)}<\infty.
\]
We set
\[
\|\theta\|_{C^{\log^2}(\T^d)} \coloneqq \|\theta\|_{L^{\infty}(\T^d)}+ [\theta]_{C^{\log^2}(\T^d)}.
\]

We prove a $\log^2$-H\"older estimate on the solution of a continuity equation driven by a divergence-free $\log$-Lipschitz vector field.

\begin{lemma}\label{L11}
Suppose that $ v$ is a divergence-free vector field, which is log-Lipschitz uniformly in time and let $X$ be the corresponding flow, i.e., 
\[
\begin{cases}
    \pt X(t,x) = v(t,X(t,x)), & t >0, \\
    X(0,x) = x, & x \in \T^d.
\end{cases}
\]
Then there exists a constant $C>0$ (depending on $v$) such that
\begin{equation}\label{30}
|X(t,x) - X(t,y)|\ge |x-y|^{\exp(Ct)},
\end{equation}
for any $t >0$, $x,y\in\T^d$ such that $|x-y|<1$. Moreover, if $\theta$ is transported by the flow, i.e., $\theta(t,\cdot)\circ X(t,\cdot)  =\theta_0$, then
\[
\|\theta(t,\cdot)\|_{C^{\log^2}(\T^d)}\lesssim e^{2Ct} \|\theta_0\|_{C^{\log^2}(\T^d)}, \qquad \text{ for $t >0$}.
\]
\end{lemma}
\begin{proof}
    For the first estimate, it is enough to assume that $|X(t,x)-X(t,y)|<1$, because, otherwise, there is nothing to prove. In this case, the estimate follows immediately from the $\log$-Lipschitz continuity, which in turn yields
    \[
    \left|\frac{\dd}{\dd t} |\log |X(t,x)-X(t,y)||\right| \lesssim |\log |X(t,x)-X(t,y)||,
    \]
    and from Gr\"onwall's lemma.

    The estimate on the solution to the transport equation is a direct consequence. Indeed, because $X^t$ is a homeomorphism on $\T^d$ and the maximal value of $\theta_0$ is preserved, it holds that
    \begin{align*}
    \|\theta(t,\cdot)\|_{C^{\log^2}(\T^d)}  & \lesssim \|\theta(t,\cdot)\|_{L^{\infty}(\T^d)} + \sup_{|x-y|\le 1/2} |\log|X(t,x)-X(t,y)||^2 |\theta(t,X(t,x))-\theta(t,X(t,y))|\\
    & = \|\theta_0\|_{L^{\infty}(\T^d)} + \sup_{|x-y|\le 1/2} |\log|X(t,x)-X(t,y)||^2 |\theta_0(x)-\theta_0(y)|\\
    &\lesssim \|\theta_0\|_{C^{\log^2}(\T^d)} \sup_{|x-y|<1/2} \frac{\log^2|X(t,x) - X(t,y)|}{\log^2|x-y|}.
    \end{align*}
    The stated bound follows from \cref{30}.
\end{proof}

Finally, we prove the following bound. 

\begin{lemma}\label{lem log holder}
Let $\theta\in C^{\log^2}(\T^2)$ be mean-zero, then it holds that \begin{align}\label{29}
\norm{\nabla^2\Delta^{-1}\theta}_{L^\infty(\T^2)}\leq \norm{\theta}_{C^{\log^2}(\T^2)}\end{align}
\end{lemma}

\begin{proof}
Thanks to the triangle inequality and the Littlewood--Paley decomposition $\theta = \sum_{k\ge 1}\theta\ast\phi_{k}$, it suffices to estimate the sum
\begin{align}
\sum_{k\geq 1}\norm{\nabla^2\Delta^{-1}\theta*\phi_k}_{L^\infty(\T^2)}\label{Besov bd}
\end{align}
against the right-hand side of \cref{29}. We start by noting that
\[
\grad^2\Delta^{-1}\theta*\phi_k = \theta*\phi_k*\grad^2 \Delta^{-1}\left(\phi_k+\phi_{k-1}+\phi_{k-2}\right),
\]
by the virtue of \cref{27}, and thus, via Young's convolution inequality and the fact that Littlewood--Paley multipliers (and thus also the $\grad^2 \Delta^{-1}\phi_k$'s) are defined as $L^1$ dilations, it holds that
\[
\|\grad^2\Delta^{-1}\theta*\phi_k\|_{L^{\infty}(\T^2)}\le \|\theta*\phi_k\|_{L^{\infty}(\T^2)}\|\grad^2 \Delta^{-1}\left(\phi_k+\phi_{k-1}+\phi_{k-2}\right)\|_{L^1(\R^2)} \lesssim \|\theta*\phi_k\|_{L^{\infty}(\T^2)}.
\]
It thus suffices to estimate $\sum_{k\ge 1}\|\theta*\phi_k\|_{L^{\infty}(\T^2)}$ against the right-hand side of \cref{29}.

We fix $x\in\T^2$ and use the mean-freeness of $\phi_k$  to the effect that
\begin{align*}
|\theta*\phi_k(x)|&\leq \left|\int_{\R^2} (\theta(x-y)-\theta(x))\phi_k(y)\dd y\right|\\
&\lesssim \norm{\theta}_{L^\infty(\T^2)}\int_{\R^2\backslash B_{2^{-k/2}}(0)}|\phi_k(y)|\dd y \\ &\qquad +  \sup_{|y|\leq \frac{1}{2}}\left(|\log |y||^2\norm{\theta(\cdot+y)-\theta}_{L^\infty(\T^2)}\right) \int_{B_{2^{-k/2}}(0)}| \log |y||^{-2}|\phi_k(y)|\dd y.
\end{align*}
We notice, on the one hand, by scaling that
\[
\int_{\R^2\setminus B_{2^{-k/2}}(0)} |\phi_k(y)|\dd y = \int_{\R^2\setminus 2^{k/2}(0)} |\phi_1(z)|\dd z \le 2^{-k} \int_{\R^2} |z|^2|\phi_1(z)|\dd z \lesssim 2^{-k},
\]
where the convergence of the integral in the last inequality follows from the fact that $\phi_1$ is a Schwartz function. On the other hand, we have that $|\log |y|| \gtrsim k$ for $k\in B_{2^{-k/2}}(0)$, combining these observations, we arrive at 
\[ 
|\theta*\phi_k(x)| \lesssim 2^{-k} \|\theta\|_{L^{\infty}(\T^2)} + \frac1{k^2} \sup_{|y|\leq \frac{1}{2}}|\log |y||^2\norm{\theta(\cdot+y)-\theta}_{L^\infty(\T^2)} .
\]
Summation over $k$ gives the result.
\end{proof}

\subsection{General setting}
\label{ssec:setting}

We recall that $f \in L^1_{\mathrm{loc}}(\T^d)$ belongs to $\mathrm{BMO}(\T^d)$, the space of functions of \emph{bounded mean oscillation}, if 
\[
\|f\|_{\mathrm{BMO}(\T^d)}\coloneqq \sup_{x\in\T^d,\,r>0} \fint_{B_r(x)}\left|f(z)-\fint_{B_r(x)} f(y) \, \d y \right| \, \d z<\infty.
\]

We will prove \cref{th:transport} assuming that the velocity field satisfies the following conditions:
\begin{gather}
 \operatorname{div} u = 0, \label{ass:div-free} \\ 
      \Big(\sum_{k\ge 1} |\nabla u_k(t,\cdot)|^2 \Big)^{\frac{1}{2}} \in L^1((0,+\infty);L^{\exp}(\T^d)), \label{ass:v-1}\\
     \Big(\sum_{k\ge 1}  2^{-2k} |\grad^2 u_k(t,\cdot)|^2\Big)^{\frac12}   \in L^1((0,+\infty);L^{\exp}(\T^d)), \label{ass:v-2}\\
     \nabla u(t,\cdot) \in L^1((0,+\infty);L^{\exp}(\T^d)), \label{ass:v-3}
     \end{gather}
    where the frequency blocks $u_k$ are defined in the same way as in \cref{ssec:lp}. To simplify the notation in the following, we set
    \[
      w(t) \coloneqq   \Big(\sum_{k\ge 1} |\nabla u_k(t,\cdot)|^2\Big)^{\frac{1}{2}} + \Big(\sum_{k\ge 1}  2^{-2k} |\grad^2 u_k(t,\cdot)|^2\Big)^{\frac12}  + |\nabla u (t,\cdot)|  .
    \]
and we notice that the hypotheses \eqref{ass:v-1}, \eqref{ass:v-2} and \eqref{ass:v-3} imply that
  \begin{align}    \label{ass:v-t}
 w \in L^1((0,+\infty),L^{\exp}(\T&^d)) 
    \end{align}

Conditions \crefrange{ass:v-1}{ass:v-t} follow from the assumption $\nabla u \in L^1((0,+\infty);\BMO(\T^d)$ used in the statement of \cref{th:transport}, as it follows from the Lemma below.

\begin{lemma}\label{lm:bmo-consequences}
Let $v\in \BMO(\T^d)$, let $\chi_k=2^{kd}\chi_1(2^k\cdot)\in \mathcal{S}(\R^d)$ be a family of mean-free functions, then 
\begin{align}
\norm{v}_{L^{\exp}(\T^d)}&\lesssim \norm{v}_{\BMO(\T^d)}\\
\norm{\left(\sum_{k\ge 1} |v*\chi_k|^2\right)^{1/2}}_{L^{\exp}(\T^d)}&\lesssim \norm{v}_{\BMO(\T^d)}.
\end{align}
\end{lemma}

In particular, from \cref{lm:bmo-consequences}, we deduce that, if $ u= K*\omega$ ($K$ being the Biot--Savart kernel), with $\omega(t,\cdot)\in L^\infty(\T^2)$, then  
 \begin{align}
 \|w(t,\cdot )\|_{L^{\exp}(\T^2)}\lesssim \norm{\omega(t,\cdot)}_{L^\infty (\T^2)},\qquad t \ge 0.
 \end{align}

\begin{proof}
The first statement is a special case of the classical John--Nirenberg inequality (see, e.g., \cite[Chapter IV, Section 1.3]{MR1232192}).

For the second result, we approximate the periodic extension of $v$ with $\zeta_n v$ where $\zeta_n$ are compactly supported smooth functions converging to $1$. Then we can check that $\sum_{k\ge 1} |(\zeta_n v)*\chi_k|^2\rightarrow \sum_{k\ge 1} |v*\chi_k|^2$ locally uniformly and that $\zeta_n v$ is still in $\BMO(\R^d)$.

By the \emph{Fefferman--Stein decomposition}\footnote{ We introduce the cut-off extension to the whole space specifically to avoid reproving a statement of this result on the torus.} (see \cite[Theorem 3]{MR447953}), for every $f\in \BMO$, vanishing at $\infty$, there exist $f_0,\dots f_d$ such that \begin{align*}
f=f_0+R_1 f_1+\dots+ R_d f_d,
\end{align*}
where $R_1\dots R_d$ are the \emph{Riesz transforms}, defined for test functions through $\widehat{(R_j \varphi)}=-i \frac{\xi_j}{|\xi|}\widehat{\varphi}$, and $\norm{f_0}_{L^\infty(\R^d)}+\dots+\norm{f_d}_{L^\infty(\R^d)}\lesssim \norm{f}_{\BMO(\R^d)}$.

Now let $v^0,\dots v^d$ be such functions for $\psi_n v$, then it holds that \begin{align*}
\chi_k*(R_i v^i)=(R_i \chi_k) *v^i.
\end{align*}
We observe that, by the homogeneity of the Riesz transform, $R_i \chi_k=2^{kd}(R_i\chi_1)(2^{k}\cdot)$. Let us recall that, by the Littlewood--Paley theorem (see \cite[Theorem 6.1.2]{MR3243734}), 
\begin{align}
    \norm{\left(\sum_{k\geq 1} |f*\tilde{\chi}_k|^2\right)^\frac{1}{2}}_{L^p(\R^d)}\lesssim p \norm{f}_{L^p(\R^d)}\label{littlewood}
\end{align}
uniformly in $p\geq 2$, for any  family $\tilde{\chi}_k=2^{kd}\tilde{\chi}_1(2^k\cdot)$ of mean-free Schwartz functions. Now $v^i$ might not be in $L^p(\R^d)$, but, as \begin{align*}
\sum_{k\geq 1}\left|\int_{\R^d\backslash [-\pi,\pi]^d} \tilde{\chi}_k(x)\dd x\right|^2\lesssim \sum_{k\geq 1} 2^{-k}\lesssim 1,
\end{align*}
we let $\tilde{\chi}_k=R_i \chi_k$ and apply \cref{littlewood} to $I_{[-2\pi,2\pi]^d}v_k$ to deduce together with Young's inequality that  
\begin{align*}
 \norm{\left(\sum_{k\geq 1}|(\zeta_nv)*\chi_k|^2\right)^\frac{1}{2}}_{L^p([-\pi,\pi]^d)}&\lesssim \sum_{j=1}^d p\norm{v^j}_{L^p([2\pi,-2\pi]^d)}+\norm{I_{\R^d\backslash [-2\pi,2\pi]^d}v^j}_{L^\infty(\T^d)}\\ &\lesssim p\norm{\zeta_n v}_{\BMO(\T^d)}.
\end{align*}
Plugging this estimate into the  series expansion of the exponential, we obtain 
\begin{align*}
\int_{[-\pi,\pi]^d} \exp{}\left(\beta\left(\sum_{k\geq 1}|(\zeta_n v)*\chi_k|^2\right)^\frac{1}{2}\right)\dd x  
 &\le \sum_{p\ge 2} \frac{p^p}{p!} \left(C_d\beta \|v\|_{\BMO(\T^d)}\right)^p+1,
\end{align*}
where $C_d$ is some numerical constant depending on $d$ and the family of $\chi_k$'s.
Using Stirling's formula in the rough form $\frac{p^p}{p!}\lesssim e^p$, we observe that the series converges provided that $\beta$ is small enough so that $C \|v\|_{\BMO(\T^d)} \beta <1/e$.
\end{proof}

\section{Propagation of regularity for the transport equation}
\label{sec:proof-hyp}

In order to slightly simplify the notation, we introduce the multiplier
\[
m(k) \coloneqq \exp\left(2s k^a\right)
\]
that occurs in the definition of our Besov space $B^{\exp,a,s}(\T^d)$. Accordingly, in view of the proof of \cref{lm:besov-equiv}, the equivalent space $\mathbf{B}^{\exp,a,s}(\T^d)$ is the one associated with its derivative $m'(k)\approx k^{a-1}m(k)$. 

In order to analyze how the exponential Besov norms \cref{eq:besov-filtered-n} behave under the evolution, we mollify the transport equation \cref{eq:transport} with the low-pass filters $\psi_{k-1}$, and study the differences $\theta_k^{\ge} = \theta-\theta\ast \psi_{k-1}$. It is customary to write the resulting equation as a transport equation 
with  a forcing term,
\begin{align}\label{eq:t-ge}
\partial_t \theta_k^\geq +u\cdot \nabla \theta_k^\geq = [ \psi_{k-1}\ast,u\cdot ]\nabla\theta,
\end{align}
where  the term on the right-hand side is the commutator of the operations ``multiply by $u$'' and ``convolve with $\psi_{k-1}$'', that is,
\[
[\psi_{k-1}\ast,u\cdot]\nabla\theta(x) \coloneqq  \intR \theta(x-y) (u(x)-u(x-y))\cdot\nabla\psi_{k-1}(y)\dd y.
\]
Thanks to the regularity of $u$ and the smoothing effect of the convolution, the transport equation \cref{eq:t-ge} fits well into the theory of DiPerna and Lions \cite{DiPernaLions89}, and thus, $\theta_k^{\ge}$ is in fact a \emph{renormalized solution}.

\begin{proof}[Proof of \cref{th:transport}] 
Our aim is to show that a suitable differential inequality for $\norm{\theta(t,\cdot)}_{{\mathbf{B}}^{\exp, a,s}(\T^d)}^2$ holds, which then shows the theorem by Lemma \ref{lm:besov-equiv}.

Motivated by the approach in \cite{2203.10860}, we start from \cref{eq:t-ge} and compute
\begin{align*}
&\frac12\frac{\text{d}}{\text{d}t}\sum_{k\geq 1}m'(k)\intT (\theta_k^\ge)^2\dd x  \\
& =  \sum_{k\geq 1} m'(k) \intT \theta_k^{\ge} [\psi_{k-1}*,u\cdot]\grad \theta\dd x\\
&=\sum_{k\geq 1} m'(k)  \intT\intR \theta_k^\geq(x)\theta(x-y) u_{k+3}^\geq(x)\cdot\grad\psi_{k-1}(y)\, \dd y \dd x\\
&\qquad-\sum_{k\geq 1} m'(k) \intT\intR \theta_k^\geq(x)\theta(x-y)u_{k+3}^{\geq }(x-y)\cdot\grad\psi_{k-1}(y)\, \dd y\dd x\\
&\qquad+\sum_{k\geq 1} m'(k)  \intT\intR \theta_k^\geq(x)\theta_{k+4}^{\le}(x-y) \left(u_{k+2}^\leq(x)- u_{k+2}^{\leq }(x-y)\right)\cdot\grad\psi_{k-1}(y)\, \dd y\dd x\\
&=:\,\mathrm{I_1} - \mathrm{I_2} + \mathrm{I_3}.
\end{align*}
Notice that, here, we have eliminated a number of terms by invoking orthogonality properties that result from the narrow-bandedness of the Littlewood--Paley projections: this can be seen by restricting to non-zero phase blocks after executing Parseval's identity (cf. \cite{2203.10860} for details).

We shall estimate the terms $\mathrm{I}_1$, $\mathrm{I}_2$, and $\mathrm{I}_3$ separately.

\underline{Step 1.} \emph{Estimate for $\mathrm{I}_1$.} We start by writing 
\begin{align}
\mathrm{I_1} &= \sum_{n\ge 3} 2^{-n} \mathrm{J}_1(n),\quad \mbox{where }\mathrm{J}_1(n)\coloneqq \sum_{k\geq 1}^\infty m'(k)  \int_{\T^d} \theta_k^\geq (2^{k+n}u_{k+n}) \cdot \left(\theta*(2^{-k}\nabla\psi_{k-1})\right) \dd x.\label{eq:I1}
\end{align}
Using now first the Cauchy--Schwarz inequality and then the H\"older-type inequality \cref{eq:hi}, we estimate 
\begin{equation}
    \label{1}
\begin{aligned}
 |\mathrm{J}_1(n)|
 &\lesssim  \int_{\T^d} \Big(\sup_{k\ge 1}n(k) |\theta_k^\geq |\Big) \Big(\sum_{k\ge 1} |2^{k+n}u_{k+n}|^2 \Big)^{\frac12} \Big(\sum_{k\ge1}k^{-a}m(k)| \theta*(2^{-k}\nabla\psi_{k-1})|^2\Big)^{\frac12}\dd x\\
&\lesssim \Big\|{\sup_{k \ge 1} n(k) |\theta_k^\geq|}\Big\|_{L^2\log L(\T^d) }  \Big\|{\Big(\sum_{k\geq 1} |2^{k+n}u_{k+n}|^2\Big)^{\frac{1}{2}}}\Big\|_{L^{\exp}(\T^d)}  \\ & \qquad  \times \Big\|{\Big(\sum_{k\geq 1}k^{-a}m(k) |\theta*(2^{-k}\nabla\psi_{k-1})|^2\Big)^\frac{1}{2}}\Big\|_{L^2\log L(\T^d)} ,
\end{aligned}
\end{equation}
where $n(k)=k^{\frac32a-1} \exp\left(sk^a\right)$.
For the first term on the right-hand side, we use that $a\le 1/2$ and apply \cref{lm:1}.
The bound of the second term is already assured by \cref{ass:v-1}, and the control of the third term is provided by   \cref{lm:interp1}.  We thus conclude that
\[
|\mathrm{I}_1| \lesssim  \|w(t,\cdot)\|_{L^{\exp}} \|\theta\|_{B^{\exp,a,s}(\T^d)}\left(\|\theta\|_{B^{\exp,a,s}(\T^d)} + \|\theta\|_{L^{\infty}(\T^d)}\right),
\]
thanks to the convergence of the geometric series. 

\underline{Step 2.}  \emph{Estimate for $\mathrm{I_2}$.} This term can be treated very similarly to the previous one. Indeed, we first rewrite $\mathrm{I}_2$ as 
  \[
  \mathrm{I}_2 = \sum_{k\ge 1} m'(k)\int_{\T^d} \theta u_{k+3}^{\ge} \cdot (\theta_k^{\ge} \ast\grad\psi_{k-1})\dd x,
  \]
and then observe that $\theta_k^{\ge}$ and $\psi_{k-1}$ have only a small frequency overlap, so that  $\theta_k^{\ge} \ast\grad\psi_{k-1} =\theta_k  \ast\grad\psi_{k-1}$. 

Applying Parseval's identity and comparing the supports of the Fourier transforms of the other two terms of the integrand, we see that the low-frequency part in  $\theta = \theta_{k+1}^{\ge} + \theta_k^{\le}$ can be dropped. 
We are thus left with 
\begin{align*}
    \mathrm{I}_2= \sum_{k\ge 1}  m'(k) \intT \theta_{k+1}^{\ge}  u_{k+3}^{\ge} \cdot (\theta *\phi_k* \grad\psi_{k-1}) \dd x,
\end{align*}
which can be estimated in the same way as $\mathrm{I}_1$. 

\underline{Step 3.} \emph{Estimate for $\mathrm{I}_3$.} 
The (initial) strategy for estimating this term is again inspired by our previous work \cite{2203.10860}. We write the inner integral with the help of the fundamental theorem and Fubini's theorem as
\begin{align*}
& \int_{\R^d}   \theta_{k+4}^{\le}(x-y) \left(u_{k+2}^{\le}(x) - u_{k+2}^{\le}(x-y)\right)\cdot \grad\psi_{k-1}(y)\dd y\\
& = \int_0^1 \int_{\R^d}    \theta_{k+4}^{\le}(x-y)\grad u_{k+2}^{\le}(x-sy):\grad\psi_{k-1}(y)\otimes y\dd y\dd s\\
& = \grad u_{k+2}^{\le}(x) : \int_0^1 \int_{\R^d}   \theta_{k+4}^{\le}(x-y) \grad\psi_{k-1}(y)\otimes y\dd y \dd s\\
&\quad + \int_0^1 \int_{\R^d}    \left(\theta_{k+4}^{\le}(x-y)-\theta_{k+4}^{\le}(x)\right)\left(\grad u_{k+2}^{\le}(x-sy) - \grad u_{k+2}^{\le}(x)\right):\grad\psi_{k-1}(y)\otimes y\dd y\dd s\\
&\quad +\theta_{k+4}^{\le}(x) \int_0^1 \int_{\R^d}     \left(\grad u_{k+2}^{\le}(x-sy) - \grad u_{k+2}^{\le}(x)\right):\grad\psi_{k-1}(y)\otimes y\dd y\dd s.
\end{align*}
Using the fundamental theorem for the second term and the observation that the last term is vanishing after an integration by parts because $u$ is divergence-free by assumption, we have
\begin{align*}
& \int_{\R^d}   \theta_{k+4}^{\le}(x-y) \left(u_{k+2}^{\le}(x) - u_{k+2}^{\le}(x-y)\right)\cdot \grad\psi_{k-1}(y)\dd y\\
& = \grad u_{k+2}^{\le}(x) : \int_0^1 \int_{\R^d}   \theta_{k+4}^{\le}(x-y) \grad\psi_{k-1}(y)\otimes y\dd y \dd s\\
&\quad + \int_0^1 \int_0^1 \int_0^1 \int_{\R^d}   y\cdot  \grad\theta_{k+4}^{\le}(x-ry) \grad^2 u_{k+2}^{\le}(x-sty)  :\grad\psi_{k-1}(y)\otimes y\otimes y\dd y\dd s \dd t \dd r\\
&=: g^1_k(x) + g^2_k(x).
\end{align*}

We begin with the estimate of that portion of $\mathrm{I}_3$ that involves $g_k^1$. Using that the non-diagonal part of $\Phi_k(y) = \grad(y \psi_{k-1}(y) )$ vanishes on frequencies outside of $B_{2^{k-1}}(0)\setminus B_{2^{k-1}}(0)$ and that the diagonal part vanishes in the product with $\nabla u$ due to the divergence-freeness, we  write $g_k^1$ as
\[
g_k^1  = \grad u_{k+2}^{\le } : \theta_{k+1}^{\le} \ast \Phi_k = \grad u_{k+2}^{\le} : \left(\theta_k+\theta_{k-1}\right) \ast \Phi_k.
\]
Exploiting cancellations in frequency space via an application of Parseval's indentity, we may then decompose
\begin{equation}\label{7}
\begin{aligned}
 \sum_{k\ge 1} m'(k) \int_{\T^d} \theta_k^{\ge} g_k^1\dd x  & =  \sum_{k \ge 1}m'(k) \int_{\T^d} \left(\theta*\chi_{k}^4\right)\,\nabla u:(\theta*\chi_{k-1}^1 *\Phi_k) \dd x \\
&\qquad  - \sum_{k\geq 1} m'(k)\int_{\T^d} \left(\theta*\chi_{k}^4\right)\,\left(\nabla u\ast \chi_{k+3}^3\right):(\theta*\chi_{k-1}^1 *\Phi_k) \dd x ,
\end{aligned}
\end{equation}
where $\chi_k^j \coloneqq  \phi_k+\dots \phi_{k+j}$. To estimate the first term on the right-hand side, we use the H\"older inequality, first in $k$ and then in the  variant \cref{eq:hi} in $x$, 
\begin{align*}
\MoveEqLeft\sum_{k \ge 1}m'(k) \int_{\T^d} \left(\theta*\chi_{k}^4\right)\,\nabla u:(\theta*\chi_{k-1}^1 *\Phi_k) \dd x\\
&\le \intT |\grad u| \Big( \sum_{k\ge 1}m'(k) |\theta*\chi_k^4|^2\Big)^{\frac12} \Big(\sum_{k\ge 1} m'(k)|\theta*\chi_{k-1}^1* \Phi_k|^2\Big)^{\frac12}\, \dd x\\ 
& \le \|\grad u\|_{L^{\exp}(\T^d)} \Bigg\|\Big( \sum_{k\ge 1}  m'(k) |\theta*\chi_k^4|^2\Big)^{\frac12}\Bigg\|_{L^2\log L(\T^d)}  \Bigg\|\Big(\sum_{k\ge 1} m'(k) |\theta*\chi_{k-1}^1* \Phi_k|^2\Big)^{\frac12}\Bigg\|_{L^2\log L(\T^d)} ,
\end{align*}
and noticing that $m'(k)\sim k^{a-1} m(k)$ and $a-1 \le -a$ for $a\le 1/2$, we see that the Besov norms can be controlled via \cref{lm:interp1}. Using assumption \cref{ass:v-3}, we thus conclude that
\begin{align*}
\MoveEqLeft \sum_{k \ge 1}m'(k) \int_{\T^d} \left(\theta*\chi_{k}^4\right)\,\nabla u:(\theta*\chi_{k-1}^1 *\Phi_k) \dd x\\
&\lesssim \|w(t,\cdot)\|_{L^{\exp}(\T^d)} \left(\|\theta\|_{B^{\exp,a,s}(\T^d)}^2 + \|\theta\|_{L^{\infty}(\T^d)}^2\right).
\end{align*}

We can deal with the second term in \cref{7} essentially in the same manner as with $\mathrm{I}_1$. We omit the details.

It remains to consider the contribution of $g_k^2$. We first notice (as in \cite[p. 15]{2203.10860}) that 
\begin{align*}
 |g_k^2(x) |
&\le   \sum_{j\ge -4} 2^{-j}\sum_{n\ge -2}2^{-n} \int_0^1\int_0^1\int_0^1\intR 2^{j-k}|\grad\theta_{k-j}(x-ry)|\\
&\qquad  \times 2^{n-k} |\grad^2 u_{k-n}(x-sty)|   \rho_k(y)\, \dd y\dd r\dd s\dd t,
\end{align*}
where we have set $\rho_k(y) \coloneqq  4^k|y|^3 |\grad\psi_{k-1}(y)|$ and used the convention that negatively indexed quantities are zero.
Applying the Cauchy--Schwarz inequality in $k$ and using that $\sum_{l\geq 1}(j+l)^{-a}m(j+l)2^{-2(j+l)}\lesssim j^{-a}m(j)2^{-2j}$ and the fact that $m'(k)\sim k^{a-1}m(k)$, we then obtain 
\begin{align*}
 \sum_{k\ge 1}  m'(k) \int_{\T^d} \theta_k^{\ge} g_k^1(x)\dd x
 &\lesssim  \int_0^1\int_0^1\int_0^1\intR \int_{\T^d}\left(\sup_{k\ge 1} n(k) |\theta_k^{\ge}(x)|\right)\Big( \sum_{k\ge 1}k^{-a}m(k) 2^{-2k}|\grad\theta_{k}(x-ry)|^2\Big)^{\frac12}\\
&\qquad   \times\Big(\sum_{k\ge 1}  2^{-2k} |\grad^2 u_{k}(x-sty)|^2\Big)^{\frac12}   \rho_k(y)\, \dd y\dd r\dd s\dd t\dd x,
\end{align*}
where $n(k)$ was defined in the estimate of $\mathrm{I}_1$. An application of the H\"older-type inequality \cref{eq:hi}, the translation invariance in $x$ then gives
\begin{align*}
 \sum_{k\ge 1}  m'(k) \int_{\T^d} \theta_k^{\ge} g_k^1(x)\dd x
&\lesssim \big\|\sup_{k\ge 1} n(k)|\theta_k^{\ge}|\big\|_{L^2\log L(\T^d)} \Bigg\| \Big( \sum_{k \ge 1} k^{-a}m(k)  2^{-2k}|\nabla\theta_{\ell}|^2\Big)^{\frac12}\Big\|_{L^2\log L(\T^d)}\\
&\qquad  \times \Bigg\|\Big(\sum_{k\ge 1}  2^{-2k} |\grad^2 u_k|^2\Big)^{\frac12} \Bigg\|_{L^{\exp}(\T^d)} \|\rho_k\|_{L^1(\R^d)}.
\end{align*}
We finally notice that $\|\rho_k\|_{L^1(\R^d)} = \|\rho_1\|_{L^1(\R^d)} \sim 1$, and an application of assumption \cref{ass:v-2} and  \cref{lm:1} and \cref{lm:interp1} gives the same estimate as for $\mathrm{I}_1$, because $a\le 1/2$.

\underline{Step 4.} \emph{Gr\"onwall-type argument and conclusion.} In conclusion, putting the previous estimates together, we get 
\begin{align*}
\frac{1}{2}\frac{\text{d}}{\text{d}t}\norm{\theta(t,\cdot)}_{{\mathbf{B}}^{\exp, a,s}(\T^d)}^2 \lesssim \|w(t,\cdot)\|_{L^{\exp}(\T^d)} \left( \norm{\theta}_{{\mathbf{B}}^{\exp, a,s}(\T^d)}^2 + \|\theta\|_{L^{\infty}(\T^d)}^2\right).
\end{align*}
Applying Gr\"onwall's inequality and the norm equivalence in Lemma \ref{lm:besov-equiv}, we conclude the proof.

\end{proof}

\section{Construction of the counterexample}
\label{sec:proof-counterexample}

The strategy of the proof is to perturb a steady state solution to the Euler equations, namely the so-called \emph{Bahouri--Chemin vortex patch} (cf.  \cite{MR1288809}):
\begin{align}
&\bar \omega(x_1,x_2)\coloneqq \sgn{\sin(x_1)}\sgn{\sin(x_2)}\label{odd odd}.
\end{align}

From the Biot--Savart representation \(\bar u=\nabla^\perp\Delta^{-1}\bar \omega\) of the (divergence-free) velocity field $\bar u$ associated with $\bar \omega$, it follows that the first component $\bar u_{1}$ is even in $x_2$ and odd in $x_1$, while the second component $\bar u_2$ is even in $x_1$ and odd in $x_2$. We observe that $\bar \omega$ is bounded and belongs to $H^{\frac{1}{2}-\varepsilon}(\T^2)$, for every $\varepsilon >0$ (which can be checked by a direct computation using the Gagliardo representation of the $H^{\frac{1}{2}-\eps}$-seminorm).

In \cref{lemma u0}, we list several elementary decay and growth estimates for $\bar u$.

\begin{lemma}[Estimates on the velocity field $\bar u$]\label{lemma u0}
Let $x=(x_1,x_2)\in (0,\frac{1}{2})^2$ with $x_2\leq x_1$. Then the following estimates hold:
\begin{align}
 |\bar u_1(x)|&\lesssim |x_1|\log\frac{1}{x_1},\label{u1 est}\\
 |\bar u_2(x)|&\lesssim |x_2|\log\frac{1}{x_2},\label{u2 est}\\
  |\nabla \bar u(x)|& \lesssim \log\frac1{|x|} ,\label{bd grad}\\
 |\nabla \de_1 \bar u(x)|&\lesssim \frac{1}{|x|}.\label{bd nabla2}
\end{align}
Furthermore, there exists $\eps>0$ such that, if  $x_2\leq \eps x_1$, then  
\begin{align}
 \de_1 \bar u_1(x) \leq - C \log \frac1{x_1},\label{u1 est 2}
\end{align}
for some $C>0$.
\end{lemma}

\begin{proof}[Proof of \cref{lemma u0}]
The estimates \cref{u1 est} and \cref{u2 est} on $\bar u_1$ and $\bar u_2$ are a  consequence of the log-Lipschitz continuity and of $\bar u_{1/2}$ being odd in $x_{1/2}$ (which, in turn, implies 
 $\bar u_1(0,x_2)=0=\bar u_2(x_1,0)$).

For the estimate \cref{u1 est 2}, it suffices to show the estimate \cref{bd nabla2} on the second derivative, as the case $x_2=0$ is already contained in \cite[Proposition 2.1]{MR1288809} up to minor modifications.

For the other estimates on the derivatives, we use that the Biot--Savart law can be written as a convolution in the sense that \begin{align}
 \nabla \bar u(x)= \mathrm{P.V.}\,\int_{[-\pi,\pi]^2}\nabla_x\frac{(x-y)^\perp}{|x-y|^2}\bar \omega(y)\dd y+\sum_{q\in2\pi  \Z_{\neq 0}^d}\int_{[-\pi,\pi]^2}\nabla_x\frac{(x+q-y)^\perp}{|x+q-y|^2}\bar \omega(y)\dd y,\label{sum est}
\end{align}
for any mean-free vorticity $\bar \omega$ on the torus. Here the infinite sum converges because the integrals can be estimated as $\lesssim |q|^{-3}\norm{\bar \omega}_{L^\infty(\T^2)}$ due to the mean-freeness of the vorticity. Similarly, we can also estimate the far-away contributions to $\nabla \de_1 \bar u$ with $\norm{\omega}_{L^\infty(\T^2)}$.
In particular, we only need to estimate the singular integral for the remaining estimates.

Arguing with a smooth approximation, we can deduce  
\begin{align*}
 &\mathrm{P.V.}\,\int_{[-\pi,\pi]^2}\partial_{x_1}\frac{(x-y)^\perp}{|x-y|^2}\bar \omega(y)\dd y\\
 &\quad=\int_{[-\pi,\pi]^2}\frac{(x-y)^\perp}{|x-y|^2}\dd \partial_{y_1}\bar  \omega(y)- \int_{\partial [-\pi,\pi]^2}e_1\cdot n(y)\frac{(x-y)^\perp}{|x-y|^2} \bar \omega(y)\dd y,
\end{align*}
where $n$ is the outer normal and $ \partial_{y_1} \bar \omega$ is the distributional derivative of $\bar \omega$, which is the Hausdorff measure $\mathcal{H}^1$ restricted to $\{y_1=0\}$ times the sign depending on the direction of the jump. It thus holds that
\begin{align}
\left|\int_{[-\pi,\pi]^2}\frac{(x-y)^\perp}{|x-y|^2}\dd \partial_{y_1}\bar  \omega(y)
\right|\leq  \int_{\{y_1=0\}}\frac{1}{|x-y|}\dd y\lesssim |\log x_1|.
\end{align}
Because $|x-y|\ge\pi/2$ on the boundary integral, the latter is a term of order $\|\bar \omega\|_{L^{\infty}}=1.$
We thus obtain the desired bound on the $\de_1$-derivative in \cref{bd grad}. As $\de_1\bar u_{1}+\de_2 \bar u_{2}=0$ and $\de_1 \bar u_{2}-\de_2 \bar u_{1}=1$ this also implies the bounds on the $\de_2$-derivative in \cref{bd grad}. 

We may argue similarly for the second-order derivative in \cref{bd nabla2}, which leads to considering
\begin{align}
\left|\int_{\T^2}\nabla_x\frac{(x-y)^\perp}{|x-y|^2}\dd \de_{y_1
}\bar \omega(y) \right|  \le \int_{\{y_1=0\}}\frac{1}{|x-y|^2}\dd y\lesssim x_1^{-1},
\end{align}
and implies \cref{bd nabla2}.
\end{proof}

We now observe that a small perturbation of $\bar \omega$ does not disturb the corresponding velocity field too much.  

\begin{lemma}[Bounds on the velocity field associated with $\omega$] \label{L3}
Let $\omega_*$ be odd in $x_1$ and $x_2$ and suppose  $\norm{\omega_*}_{L^\infty(\T^2)} < \infty$. Then the velocity $u=\nabla^{\perp}\Delta^{-1}\omega$ associated with the solution to the vorticity equation with initial datum $\bar \omega+\omega_*$ satisfies the bounds \begin{align}
&|u_1(t,x)|\lesssim x_1\log\frac{1}{x_1}\label{est u1}\\
&|u_2(t,x)|\lesssim x_2\log \frac{1}{x_2}\label{est u2}
\end{align} globally in time for $x_1,x_2\in (0,\frac{\pi}{4})$ and is log-Lipschitz with uniformly bounded constant.
\end{lemma}
 
\begin{proof}
We first note that the vorticity equation preserves the oddness in the $x_1$ and $x_2$ variables. In particular, this means that at any time $t$, we must have \begin{align}\label{36}
u_2(t,x_1,0)=u_1(t,0,x_2)=0.
\end{align}
As the $L^\infty$-norm of $\omega$ is preserved, this gives a global bound on the log-Lipschitz norm, which implies the statement. 
\end{proof}

\subsection{Construction of the perturbation}

To prove \cref{th:sharp}, it is pivotal to choose the perturbation $\omega_\ast$ in a suitable way. The key idea (at the linear level) is that we consider a linear combination of ``blobs'' centered around $x_1=2^{-2n}$. Due to the structure of the flow, these blobs will be compressed by a factor $e^{Ctn}$ in the $x_1$-direction and will be stretched by the same factor in the $x_2$-direction, leading to an increase of the norm by a corresponding factor and, if set up with appropriate coefficients, also to a loss of regularity.

The difficulty is to keep the singular hyperbolic nature of the flow stable under perturbation at the non-linear level. Controlling just the velocity $u$ itself is quite easy if the perturbation is small and has the same symmetry (see  \cref{L3}).

For technical reasons, squeezing by a factor $n$ alone is not actually enough for an increase of the norm, for instance, we have to exclude the scenario that the flow rotates a lot and the stretching in the $x_2$- and the compression in the $x_1$-direction cancel each other out. Therefore we will also need some control on the gradient of the velocity. This will be achieved in \cref{L10} by suitably controlling the location of the ``blobs'' and by making them extremely small in amplitude.

Due to the symmetry, we only need to consider the first quadrant $(0,\pi)^2$ of the flat torus $\T^2$ and define $\omega_*$ in the other quadrants by symmetry.
We take \begin{align}
Q_n&\coloneqq(2^{-2n},3\cdot 2^{-2n})\times (2^{-2Dn},3\cdot 2^{-2Dn}),\\
R_n &\coloneqq (3\cdot 2^{-2n-1}, 2^{-2n+1})\times (3\cdot 2^{-2Dn-1}, 2^{-2Dn+1}),
\end{align}
where $D\ge 1$ is some large but fixed constant. It holds that $R_n\subset \subset  Q_n$, the larger rectangles are pairwise disjoint,  and any $Q_n$ has a distance larger than $  2^{-2n-4}$ from all other rectangles, except for the reflection of $Q_n$ on the $x_1$-axis, from which it has distance $\sim 2^{-2Dn}$. We denote by $Q$ the union of all these $Q_n$'s. We furthermore let $\zeta_n(x_1,x_2) = \zeta_0(2^{2n}x_1,2^{2Dn}x_2)$ denote a smooth cut-off function with values in $[0,1]$, that is supported on the rectangle $Q_n$ and equals $1$ in the interior rectangles $R_n$. We then introduce an oscillating function on each rectangle,
 \begin{align}
\omega_n(x)\coloneqq  h_n \zeta_n(x)\sin(d_n x_1),
\end{align} 
for suitable sequences $\{h_n\}_{n \in \N}$ and $\{d_n\}_{n \in \N}$, 
and glue these pieces together by setting 
\begin{align}
\omega_*=\sum_{n\geq n_*}\omega_n,
\end{align}
for some $n_*\in\N$. The functions $\omega_n$ are smooth and supported in $Q_n$.

Our first goal is to select the amplitudes $h_n$   and the frequencies $d_n$ to suit our regularity setting. This is done in the following lemma.

\begin{lemma}[Choice of amplitudes and frequencies] \label{lm:hndn}
Let $a\in(0,1]$ be given and let $s<\frac{1}{2}$ if $a=1$ and suppose that $h_{n+1}<h_n <1<d_n<d_{n+1}$ are such that
\begin{align}
        \label{22}
    d_n &\gg 4^{Dn},\\
        \label{23}
    h_n \log_2(d_n) \exp(s\log_2^a d_n) \sqrt{|Q_n|} &\le \frac1{n},\\
    h_n&\le \frac1{n} \label{24}.
\end{align}
Then $\omega_*\in X_{a,s,1}$, which in particular implies that $\omega_*\in B^{\exp,a,s}(\T^2)$ by \cref{pointwise desc} above.
\end{lemma}

\begin{remark}[Optimal determination of the amplitude]
   The conditions \crefrange{22}{24} that arise in the proof of \cref{lm:hndn} motivate the following maximal choice of the amplitude
   \begin{equation}
       \label{300}
       h_n\coloneqq \frac1n\min\left\{1,\frac1{\log_2(d_n) \exp(s\log_2^a d_n) \sqrt{|Q_n|}}\right\}.
   \end{equation}
   We will later see that the frequencies $d_n$ necessarily grow like $2^{(n^\delta)}$ with $\delta a>1$, so that this choice of amplitude saturates \eqref{23} for large $n$.
\end{remark}

\begin{proof}[Proof of \cref{lm:hndn}]
To shorten the notation in the following, we set $\Phi(\rho) \coloneqq \log_2\left(1+\rho\right)\exp\left(s  \log^a_2{\rho}\right)$ and we note that this function satisfies the sublinear growth formula,
\begin{equation}\label{21}
\frac{\rho_1}{\Phi(\rho_1)}\le \frac{ \rho_2}{\Phi(\rho_2)}\quad\mbox{for any }1\ll \rho_1\le \rho_2.
\end{equation}

We start by considering the case where $x,y\in Q_n $   with $|x-y|d_n \le 1$. In this case,   we use the Lipschitz continuity of $\omega_n$ on $Q_n$ and the growth formula \cref{21} with $d_n\gg1$ to show
\[
|\omega_*(x)-\omega_*(y)|  = |\omega_n(x)-\omega_n(y)|\lesssim h_n d_n |x-y| \le h_n \Phi(d_n) \Phi\left(\frac1{|x-y|}\right)^{-1}.
\]
Similarly, we obtain for $|x-y|d_n \geq 1$,  using the monotonicity of $\Phi$ itself,
\[
|\omega_*(x)-\omega_*(y)|\le |\omega_n(x)| + |\omega_n(y)|   \lesssim h_n \le h_n \Phi(d_n) \Phi\left(\frac1{|x-y|}\right)^{-1}.
\]
Likewise, if $x\in Q_n$ and $y\in Q_m$ with $m>n$, we have that $|x-y|\gtrsim 2^{-2n-4}\gtrsim 2^{-2m-4}$, and thus, thanks to \cref{22}, we can argue as before
\[
|\omega_*(x)-\omega_*(y)|\le   h_n +h_m \le \left(h_n \Phi(d_n)+h_m \Phi(d_m)\right) \Phi\left(\frac1{|x-y|}\right)^{-1}.
\]

All these bounds motivate the choice $g = h_n \Phi(d_n)$ on $Q_n$ in \cref{def g}.  

Now, if $x\in Q_n$ and $y\not \in Q$ such that $|x-y|\ge2^{-2n-5}$, the same argument as before applies, and there is no contribution to $g(y)$. 
In the case when $x$ and $y$ are in different quadrants the same argument still works as the pairwise distance is always $\gtrsim d_n,d_m$.

Finally, in the remaining case  where $x\in Q_n$ and $y\not \in Q$ are such that $|x-y|\le 2^{-2n-5}$, we write with the help of the monotonicity of $\Phi$ that
\[
|\omega_*(x)-\omega_*(y)|  \le h_n \le h_n \Phi\left(\frac1{\dist(y,Q_n)}\right)\Phi\left(\frac1{|x-y|}\right)^{-1}.
\]
We thus choose 
\[
g(y) \coloneqq \begin{cases} h_n \Phi(d_n) &\quad\mbox{for }y\in Q_n,\\
 h_n \Phi\left(\frac1{\dist(y,Q_n)}\right) &\quad \mbox{for }y\not\in Q\mbox{ such that }\dist(y,Q_n)\le 2^{-2n-5},\\
0&\quad\mbox{elsewhere}.
\end{cases}
\]
Notice that $g$ is well-defined because the distance of $Q_n$ to its neighbor rectangles is larger than $2^{-2n-4}$.
It remains to evaluate the $L^2$ norm of $g$. First on $Q$, we have that
\begin{align*}
\|g\|_{L^2(Q)}^2 \le \sum_{n\ge 1} h_n^2 \Phi(d_n)^2 |Q_n| &\lesssim \sum_{n\ge 1} \left( h_n \log_2(d_n) \exp(s \log_2^a d_n) \sqrt{|Q_n|}\right)^2 \\ &\le\sum_{n\ge 1}  \frac1{n^2}<+\infty,
\end{align*}
by our choice of $h_n$ in \cref{23}. On the complement, we have that
\[
\|g\|_{L^2(Q^c)}^2  \le \sum_{n\ge 1} h_n^2 \int_{\dist(y,Q_n)\le 2^{-2n-5}} \Phi\left(\frac1{\dist(y,Q_n)}\right)^2\dd y.
\]
Because the integrand is growing towards the boundary of $Q_n$ slower than $\dist(y,Q_n)^{1-\gamma}$ with $\gamma >0$ (here we are using that if $a=1$ then $s<\frac{1}{2}$), the integrals are bounded uniformly in $n$, and the convergence of the sum then follows by \cref{24}.
\end{proof}

\subsection{A priori estimates}
\label{ssec:ex-estimates}

To show a loss of regularity, we will need some a priori estimates on the vorticity $\omega(t,\cdot)$ obtained from the initial datum $\bar \omega + \omega_*$ and the corresponding velocity field $u=\nabla^{\perp}\Delta^{-1}\omega$. Let $X$ be the flow generated by $u$, i.e., 
\[
\begin{cases}
\partial_t X(t,x)  =u(t,X(t,x)), & t >0, \ x \in\T^2, \\
X(0,x)=x, & x \in \T^2.
\end{cases}
\]
Because of  symmetry, the trajectories do not cross the axes and thus the  flow leaves the Bahouri--Chemin vortex patch invariant during its evolution:  
\[
\omega(t,x)=\bar \omega(x)+\omega_*\circ (X(t,\cdot))^{-1}(x), \qquad t >0, \ x \in \T^2,
\]
and then 
\[
u(t,x) = \bar u(x) + u_\ast(t,x), \qquad t >0, \ x \in \T^2.
\]
In what follows, we will also use the shorter notation 
\[
\omega_*^t \coloneqq \omega_*\circ (X(t,\cdot))^{-1} \quad\text{and}\quad u^t_\ast \coloneqq \nabla^\perp \Delta^{-1} \omega_\ast^t.
\]

As a consequence of the bounds \cref{est u1} and \cref{est u2} in \cref{L3}, 
\[
X_1(t,x) \ge x_1^{\exp(Ct)},\quad X_2(t,x) \le x_2^{\exp(-Ct)},
\]
for any $x_1,x_2 \in (0,\pi/4)$, where $C>0$ is the implicit constant in \cref{est u1} and \cref{est u2}. Using the convexity of the exponential function and restricting to times $Ct\le 1$, we observe that $\exp(Ct)  \le1+ Ct \exp(Ct)\le 1+Cet$ and $\exp(-Ct)\ge 1-Ct$, and thus 
\[
X_1(t,x) \ge x_1^{1+Cte},\quad X_2(t,x) \le x_2^{1-Ct}.
\]
For $x\in Q_n$ with $x_1,x_2\in(0,\pi/4)$ and $Ct\le 1$, we then deduce that
\begin{align}
X_1(t,x)\gtrsim 2^{-(1+Cet)2n},\label{est x1}
\end{align} and \begin{align}
X_2(t,x)\lesssim  2^{-2(1-Ct)Dn}.\label{est x2}
\end{align}

Let $Q_n^t$ be the image of $Q_n$ under the flow at time $t$, that is, $Q_n^t \coloneqq X(t,Q_n)$. Due to the log-Lipschitz regularity \cref{eq:log-lip}, we can apply \cref{L11} and deduce
\[|X(t,x)-X(t,y)|\ge |x-y|^{\exp(Ct)}\] for any two sufficiently close points $x$ and $y$. Recalling that each $Q_n$ has an initial distance of the order $\gtrsim 2^{-2n}$ from its neighbors, the evolving cell $Q_n^t$  has a distance of the order $\gtrsim 2^{-2en}$ from its neighbors for $Ct\le1$ . Furthermore, if $n \ge n_*$ (chosen sufficiently large, which we will do from now on), then $Q_n^t$ has order $1$ distance from the boundary $\partial [-\pi,\pi]^2$.

\begin{lemma}[Uniform local estimate on the velocity gradient] \label{L10}
Let us suppose that $h_n$ is given by \eqref{300} and suppose that 
\begin{equation}
    \label{301}
    d_n = 2^{\delta(n)},
\end{equation}
for some monotone function $\delta(n)$ satisfying $\delta(n)^a/n\to \infty$ for $n\to \infty$ and $\delta(n)\le e^n$ for any $n\ge n_*$ sufficiently large.
Then, if $D$ is sufficiently large, we have 
\begin{align}\|\nabla u_*(t,\cdot)\|_{L^{\infty}(Q_n^t)}\lesssim 1\label{bd grad 2}.
\end{align}
In particular, 
\begin{equation}
    \label{34}
    \|\nabla u(t,\cdot)\|_{L^{\infty}(Q_n^t)}\lesssim n,
\end{equation}
and 
\begin{equation}
    \label{35}
    \de_1 u_1 \lesssim -n.
    \end{equation} 

Moreover, we have 
\begin{align}
\|\de_1 u_2(t,\cdot)\|_{L^{\infty}(Q_n^t)}\lesssim 2^{-\frac{D}{3}n}.\label{bd de12}
\end{align}
\end{lemma}

\begin{proof} 
We start by noticing that the assumptions on $\delta(n)$ entail    \cref{22} and guarantee, in addition, that
\begin{align}
    \label{31}
 h_n&\le 2^{-5en  - n},\\
 h_n \log_2^2 d_n &\le1,
\end{align}
for any $n\ge n_*$.
We split the velocity into the stationary part $\bar u$ and the perturbation $u_*^t$. We decompose $u_* = \sum_{m\ge n_*} u_m^t$ and estimate the $L^\infty$-norm of the gradient on $Q_n^t$. We deal with the diagonal part $m=n$ and the off-diagonal parts $m\not=n$ separately.

\underline{Step 1.} \emph{Proof of \cref{bd grad 2}.} For the latter, we use the Biot--Savart law as in \cref{lemma u0} to write
\begin{align*}
\grad u_m^t(x) &=\frac1{2\pi} \int_{[-\pi,\pi]^d} \grad_x \frac{(x-y)^{\perp}}{|x-y|^2}  \omega_m^t(y)\dd y\\
&\quad + \sum_{0\not=q\in 2\pi\Z^d} \frac1{2\pi} \int_{[-\pi,\pi]^d} \grad_x\left(\frac{x+q-y}{|x+q-y|^2} -\frac{x+q}{|x+q|^2}\right)^{\perp}\omega_m^t(y)\dd y,
\end{align*}
where we have used that $\omega_m^t$ has zero mean in the long-range terms. These are controlled by the sum
\[
\sum_{0\not=q\in 2\pi\Z^d} \frac1{|q|^3} \|\omega_m^t\|_{L^{1}}\lesssim h_m |Q_m^t| \lesssim 2^{-m},
\]
 by \eqref{31}. In total, the off-diagonal long-range contributions are thus order one, $\sum_{m\ge 1}2^{-m} \lesssim1 $. The off-diagonal short-range contribution is estimated by
 \[
 \int_{[-\pi,\pi]^2} \frac1{|x-y|^2}|\omega_m(y)|\dd y \lesssim \frac1{\dist^2(Q_n^t,Q_m^t)} \|\omega_m\|_{L^1} \lesssim 2^{4em} h_m |Q_m^t|.
 \]
 Our bound on $h_m$ in \eqref{31} gives that this contribution is again at most of the order $2^{-m}$.

It remains to treat the diagonal contribution. We notice that $\omega_n$ is $\log^2$-H\"older with norm
\begin{equation}
    \label{32}  [\omega_n]_{C^{\log^2}} \lesssim h_n \log^2 d_n.
\end{equation}
Indeed, we will derive this estimate via interpolation between the two trivial estimates
\[
\|\omega_n\|_{L^{\infty}} \le h_n  \quad \mbox{and}\quad \|\grad\omega_n\|_{L^{\infty}} \lesssim h_n d_n,
\]
where the gradient estimates makes use of the assumption that $d_n $ grows super-exponentially. For any $R<1/2$, we then use the $L^{\infty}$ estimate for $|y|\ge R$ to deduce that 
\[
|\omega_n(x+y)-\omega_n(x)| \log^2\frac1{|y|} \le 2 \|\omega_n\|_{L^{\infty}} \log^2\frac1{|y|} \lesssim h_n \log^2 \frac1R.
\]
Furthermore, in view of the gradient estimate and the fact that $s\mapsto s\log^2 \frac1s$ is growing for $s\lesssim 1$, we have, for $|y|\le R$,
\[
|\omega_n(x+y)-\omega_n(x)| \log^2\frac1{|y|} \le \|\grad\omega_n\|_{L^{\infty}} |y|\log^2\frac1{|y|} \lesssim h_nd_n R \log^2 \frac1R.
\]
Combining both inequalities, we find that
\[
|\omega_n(x+y)-\omega_n(x)| \log^2\frac1{|y|} \lesssim h_n \log^2\frac1R \left(1+ d_n R\right),
\]
for any $|y|\le 1/2$ and any $R\lesssim 1$. The choice $R=1/d_n$ is thus admissible and yields the claim in \cref{32}.

Because $\log^2$-H\"older regularity is propagated by $\log$-Lipschitz vector fields by \cref{L11}, the same estimate applies to the perturbation blob $\omega_n$ at any later time $t\lesssim 1$. Due to \cref{lem log holder} and the Biot--Savart law $u_n^t = \grad^{\perp}\Delta^{-1}\omega_n^t$, it then follows that
\begin{equation}
    \label{37}
\|\grad u_n^t\|_{L^{\infty}}\lesssim h_n + h_n \log^2 d_n \lesssim h_n \log^2 d_n.
\end{equation}
By our assumption on $h_n$ and $d_n$, we deduce that the right-hand side is controlled uniformly in $n$, which completes the proof of the gradient estimate \cref{bd grad 2}.

\underline{Step 2.} \emph{Proof of \crefrange{34}{35}.} We now  deduce \cref{34} and \cref{35} from \cref{bd grad 2}. In light of the flow estimate \cref{est x1}, the $x_1$ component of any $x$ in $Q_n^t$ is larger then order $2^{-2(1+e)n}$, and thus, the gradient estimate on the Bahouri--Chemin velocity \cref{bd grad} yields
\[
|\grad \bar u(x)| \lesssim |\log |x|| \lesssim |\log x_1| \lesssim n.
\]
 It follows from the just established bound \cref{bd grad 2} that
\[
\|\grad u(t,\cdot)\|_{L^{\infty}(Q_n^t)} \le \|\grad \bar u\|_{L^{\infty}(Q_n^t)}+\|\grad  u_*(t,\cdot)\|_{L^{\infty}(Q_n^t)} \lesssim n+1 \lesssim n.
\]

Similarly, given $\eps$ as in \cref{lemma u0}, if $D$ is sufficiently large, it follows from the flow estimates \cref{est x1} and \cref{est x2} that $x_s\le \eps x_1$ for all $x\in Q_n^t$ with $t\ll1$ and thus
\cref{u1 est 2} implies that
\[
\partial_1 \bar u_1(x) \le -C |\log x_1| \lesssim -n.
\]

\underline{Step 3.} \emph{Proof of \cref{bd de12}.} Thanks to the symmetry properties~\cref{36}, we have 
\begin{align}
\bar u_{2}(x_1,0)=\de_1 \bar u_{2}(x_1,0)=0,
\end{align}
and thus, via \cref{bd nabla2},
\[
|\partial_1 \bar u_2(x)| \le \sup_{\lambda\in(0,1)} |\grad \partial_1 \bar u_2(x_1,\lambda x_2)| |x_2| \lesssim \frac{x_2}{x_1}.
\]
We use again the flow estimates \cref{est x1} and \cref{est x2} to bound the components of $x\in Q_n^t$. We obtain for $Ct\le 1/2$ that
\[
|\partial_1 \bar u_2(x)| \lesssim 2^{-(D/2-2-e)n}\le 2^{-D/3n},
\]
provided that $D\ge 6(2+e)$.

Regarding $u_*^t$, we recall the splitting into diagonal and off-diagonal contributions. For the former, we recall from \cref{37}  that 
\[
|\partial_1 u_n^t(x)| \lesssim h_n \log^2 d_n,
\]
and $h_n$ is chosen small enough so that the claimed estimate applies. For the latter, we estimate, using the symmetry \cref{36} that
\[
|\partial_1 u_{m,2}^t(x)| \le \sup_{\lambda\in(0,1)} |\grad \partial_1 u_{m,2}^t(x_1,\lambda x_2)||x_2|
\]
for any $m\not=n$. Arguing as in Step 1, we find for the Hessian of $u_m^t$ the bound
\[
\|\grad^2 u_m^t\|_{L^{\infty}(S^t_n)} \lesssim \left(\frac1{\dist(S_n^t,Q_m^t)^3}+1\right) \|\omega_m^t\|_{L^1},
\]
where $S_n^t = \{(x_1,\lambda x_2):\: x\in Q_n^t,\, \lambda\in(0,1)\}$. The distance between the cell $Q_m^t$ and the strip $S_n^t$ is always larger than $\max\{2^{-2en},2^{-2em}\}$, and thus, using \cref{est x2}, we obtain
\[
|\partial_1 u_{m,2}^t(x)| \lesssim \min\{2^{6en},2^{6em}\} h_m |Q_t^m| 2^{-D/2 n},
\]
for any $x\in Q_n^t$. Under our bound \eqref{301} on $h_m$, the $m$-dependent terms are all controlled by $2^{-m}$. Summing up thus yields the desired control.
\end{proof}

\subsection{Proof of the norm inflation on every  \texorpdfstring{$Q_n^t$}{Qnt}}

We have now all the tools at hand to establish the loss-of-regularity claim of \cref{th:sharp}. For this purpose, we make the Ansatz
\[
\delta(n) = n^\delta
\]
in \eqref{301}, where $\delta>1/a$ needs still to be fixed, and we choose $h_n$ as in \eqref{300}, so that \eqref{23} is in fact an identity.

\begin{proof}[Proof of \cref{th:sharp}]
We want to show that, for some small $t>0$, we have that $\omega(t,\cdot)\notin X_{a,s,-1}$, which then proves the claim due to \cref{pointwise desc}. To this end, we will show  that if \cref{def g} is satisfied by $\omega(t,\cdot)$ with regards to the parameters $a,s,-1$, for some locally integrable function $g$, then we must have 
 \begin{align}
\norm{g}_{L^2(R_n^t)}\gtrsim \sqrt{|Q_n^t|}h_n\log_2^{-1}\left(1+ C d_n e^{Ctn}\right)\exp\left(s\log_2^a \left(\frac{d_n}C e^{\frac{t}Cn}\right)\right).\label{lower bd g}
\end{align}

\underline{Step 1.} \emph{Condition \cref{lower bd g} implies blow up.} Before proving \cref{lower bd g}, we show how it yields the blow-up. Actually, for $\omega(t,\cdot)\not\in X_{a,s,-1}$ to hold, we should see that $g\not\in L^2(\T^2)$.
By the concavity of the mapping $z\mapsto z^a$ and our choice of $d_n$, we have that for $n_*$ sufficiently large
\begin{align*}
\log_2^a\left(\frac{d_n}C e^{\frac{t}C n}\right) & = \left( \log_2d_n  +\frac{tn}C \log_2 e-\log_2 C\right)^a\\
    &\ge  \log_2^ad_n  + a \left(\log_2d_n +\frac{tn}C \log_2 e-\log_2 C\right)^{a-1}\left(\frac{tn}C \log_2 e-\log_2 C\right)\\
    &\ge \log^a_2d_n  + a 2^{a-1} \frac{tn}C \log_2^{a-1} d_n     .
\end{align*}
Our choice  $d_n = 2^{n^\delta}$  then yields
\begin{align}
\log_2^a\left(\frac{d_n}C e^{\frac{t}C n}\right)
\ge n^{\delta a} +a2^{a-1} n^{(a-1)\delta+1} \frac{t }{C}.
\end{align}
Because $d_n$ is growing super-exponentially in $n$, we use these estimates and \cref{lower bd g} to infer the bound
\[
\|g\|_{L^2(R_n^t)} \gtrsim \sqrt{|Q_n^t|} h_n n^{-\delta} \exp\left(s\log_2^a d_n\right)\exp\left(s2^{a-1} n^{(a-1)\delta+1} \frac{t}{C} \right).
\]
Thanks to the saturating choice of $h_n$ in \cref{23}, the latter becomes
\[
\|g\|_{L^2(R_n^t)}\gtrsim n^{-\gamma}\exp\left(s2^{a-1} n^{(a-1)\delta+1} \frac{t}{C} \right),
\]
for some $\gamma>0$. Hence, for any time $t>0$, the right-hand side diverges as $n\to +\infty$ if $(a-1)\delta +1 >0$. Since \cref{23} enforces that $\delta a>1$, both conditions together are compatible if and only if $\frac{a-1}a+1>0$, or, equivalently, $a>1/2$.

\underline{Step 2.} \emph{Proof of \cref{lower bd g}.} We keep $n\in\N$ fixed and recall that,  by the symmetry of the flow, $\omega(t,\cdot)=\bar \omega+\omega_*\circ (X(t,\cdot))^{-1}$. Therefore,   on $Q_n^t$ the value of $\omega$ is $1$ plus the image of $\omega_n$ under the flow.
We consider rectangles of the form 
\[
P_{j,m}= \left(d_n^{-1}\left(j+\frac{1}{4}\right)\pi,d_n^{-1}\left(j+\frac{3}{4}\right)\pi\right)\times \left(me^{-\exp({\kappa n})},(m+1)e^{-\exp({\kappa n})}\right),
\]
for some $\kappa>0$, and let $P_{j,m}^t = X(t,P_{j,m})$ be their images under the flow map.  We note that, because these rectangles are initially pairwise disjoint, they remain disjoint during their evolution. Moreover, because the flow is volume-preserving,  if $n_*$ is chosen sufficiently large, their union covers at least $1\%$ of $R_n^t$, which can be easily checked initially. We finally remark that on $P_{j,m}^t\cap R_n^t$, it holds that
$\omega_n(t)= h_n$ if $j$ is even and $\omega^n(t)=- h_n $ if  $j$ is odd. In particular, if $x\in P_{j,m}^t\cap R_n^t$ and $y\in P_{j+1,m}^t\cap R_n^t$, it holds that $\omega(t,x) - \omega(t,y) = \pm 2h_n$, and thus the estimate \cref{def g} becomes
\[
g(x) + g(y) \ge 2h_n \log_2^{-1} \left(1+\frac1{|x-y|}\right) \exp\left(s\log_2^a\frac1{|x-y|}\right).
\]

Provided that 
\begin{align}
d_n^{-1} e^{-Ctn}\lesssim \max_{x\in P_{j,m}^t\cap Q_n^t,y\in P_{(j+1),m}^t\cap Q_n^t}|x-y|\lesssim d_n^{-1} e^{-t/Cn},\label{bd rect}
\end{align}
holds, for some $C\ge 0$, from \cref{def g}, we deduce the lower bound
\[
g(x)+g(y) \ge 2 h_n \log_2^{-1}\left(1+ C d_n e^{Ctn}\right)\exp\left(s\log_2^a \left(\frac{d_n}C e^{t/Cn}\right)\right).
\]
Of course, the estimate is symmetric in $x$ and $y$, and thus, whenever $x\in P_{j,m}^t\cap R_n$ such that also $P_{j+1,m}^t$ or $P_{j-1,m}^t$ lies inside of $R_n^t$, it holds that
\[
g(x) \ge  h_n \log_2^{-1}\left(1+ C d_n e^{Ctn}\right)\exp\left(s\log_2^a \left(\frac{d_n}C e^{t/Cn}\right)\right).
\]

for at least half (in the sense of the measure) of these $x$.

An integration over all these $x$, which certainly cover $\frac{1}{200}$-th of $R_n^t$, which is in a volume comparable to $Q_n^t$, then gives \cref{lower bd g}.

\underline{Step 4.} \emph{Proof of \cref{bd rect}.} To conclude the proof, we need to show that \cref{bd rect} holds. We first use Gr\"onwall's inequality and the gradient bound \cref{34} to show that the diameter of the  set $(P_{j,m}\cup P_{(j+1),m})\cap R_n$ grows or decreases at most with an exponential rate of the order $n$, that is
\[
e^{-Cn t} |x-y| \le |X(t,x)-X(t,y)| \le e^{Cn t}|x-y|,
\]
for some $C\le 1$ and any $x \in Q_n\cap P_{j,m}$ and $y\in Q_n\cap P_{j+1,m}$. By construction and our choice of $d_n$, it must hold that
\[
d_n^{-1}\lesssim |x-y| \lesssim d_n^{-1} + e^{-\exp(\kappa n)} \lesssim d_n^{-1},
\]
and thus, the lower bound in \cref{bd rect} is already established. 

For the upper bound, we have to iterate these estimates component-wise. First, using \cref{34} again and \cref{bd de12}, we get 
\begin{align*}
\frac{\mathrm{d}}{\mathrm{d}t}|X_2(t,x)-X_2(t,y)|& \leq |X_2(t,x)-X_2(t,y)|\sup_{Q_n^t}|\de_2 u_2|+|X_1(t,x)-X_1(t,y)|\sup_{Q_n^t}|\de_1 u_2|\\
&\lesssim n |X_2(t,x)-X_2(t,y)|+d_n^{-1}e^{Ctn}2^{-D/3n}.
\end{align*} 
Using Gr\"onwall's inequality and the fact that the $x_2$-components are initially order $e^{-\exp(\kappa n)}$ apart, we thus find
\begin{equation}
    \label{40}
|X_2(t,x)-X_2(t,y)| \lesssim e^{-D/4 n}d_n^{-1}, 
\end{equation}
 if $D$ is sufficiently large.
Similarly, using \cref{35}, \cref{34}, and the previous bound, we deduce  \begin{align*}
&\frac{\mathrm{d}}{\mathrm{d}t}\left|X_1(t,x)-X_1(t,y)\right|\leq |X_1(t,x)-X_1(t,y)|\sup_{Q_n^t}\de_1 u_1+|X_2(t,x)-X_2(t,y)|\sup_{Q_n^t}|\de_2 u_1|\\
&\leq -\frac{n}{C}|X_1(t,x)-X_1(t,y)| + Cn d_n^{-1} e^{-D/4n}.
\end{align*}
Another application of Gr\"onwall's inequality then yields
\[
|X_1(t,x) - X_1(t,y)| \lesssim e^{-n/C t} d_n^{-1} + e^{-D/4 n} d_n^{-1} \lesssim e^{-n/C t} d_n^{-1}.
\]
Together with \cref{40}, this establishes the upper bound in \cref{bd rect}.
\end{proof}

\begin{remark}[Fractional Sobolev spaces]\label{rk:fractional-spaces}
As pointed out in \cref{rk:Hs-bu}, the proof works in a completely analogous way for $H^s(\T^2)$ spaces with  $s< \frac{1}{2}$: we can, e.g., take 
\begin{align}
&d_n\coloneqq  e^{(\frac{1}{10}e^{n})},
\end{align}
and choose, as before,  $h_n$ as in \eqref{300}. The restriction to $s<\frac12$ comes from the fact that the Bahouri--Chemin vortex belongs to $H^{1/2-}(\T^2)$.
\end{remark}

\vspace{0.5cm}

\section*{Statement about conflicting interests}
The authors declare that there is no conflict of interest.

\section*{Acknowledgments}

We thank M. Colombo and E. Zuazua for helpful discussions on the topic of this work. 

DM and CS are funded by the Deutsche Forschungsgemeinschaft (DFG, German Research Foundation) under Germany's Excellence Strategy EXC 2044–390685587, Mathematics M\"unster: Dynamics--Geometry--Structure. DM thanks the Friedrich-Alexander-Universität Erlangen-Nürnberg for their kind hospitality.

NDN is a member of the Gruppo Nazionale per l’Analisi Matematica, la Probabilità e le loro Applicazioni (GNAMPA) of the Istituto Nazionale di Alta Matematica (INdAM).  He has been  supported by the Swiss State Secretariat for Education, Research and Innovation (SERI) under contract number MB22.00034 through the project TENSE. He thanks the Universität Münster for their kind hospitality.

\vspace{0.5cm}

\printbibliography

\vfill 
\end{document}